\newif \ifMMMAS \MMMASfalse

\ifMMMAS

\documentclass{ws-m3as}
\usepackage[super]{cite}
\usepackage{url}

\else

\documentclass[a4paper,11pt]{article}

\fi

\usepackage{amsmath,amsfonts,amssymb,mathtools}
\usepackage{mathbbol}
\DeclareSymbolFontAlphabet{\mathbb}{AMSb}
\DeclareSymbolFontAlphabet{\mathbbl}{bbold}
\usepackage{upgreek}
\usepackage{nicefrac}
\usepackage{cancel}
\usepackage{comment}
\usepackage{enumerate}
\usepackage{graphicx}
\usepackage{xcolor}
\usepackage{paralist}
\usepackage{pgfplots}
\usepackage{subcaption}
\usepackage{etoolbox}
\usepackage[skins]{tcolorbox}
\makeatletter
\patchcmd{\ttlh@hang}{\parindent\z@}{\parindent\z@\leavevmode}{}{}
\patchcmd{\ttlh@hang}{\noindent}{}{}{}
\makeatother
\usepackage{multirow}
\usetikzlibrary{patterns}
\usepackage{hyperref}
\hypersetup{
  colorlinks = true,                        
  linkcolor = blue,                         
  citecolor = magenta,                      
  filecolor = violet,                       
  urlcolor = violet                         
}

\ifMMMAS

\catchline{}{}{}{}{}
\title{A discrete Weber inequality on three-dimensional hybrid spaces with application to the HHO approximation of magnetostatics}
\author{Florent Chave}
\address{Inria, Univ.~Lille, CNRS, UMR 8524 -- Laboratoire Paul Painlev\'e\\ F-59000 Lille, France\\\email{florent.chave@inria.fr}}
\author{Daniele A.~Di Pietro}
\address{IMAG, Univ.~Montpellier, CNRS\\ Montpellier, France\\\email{daniele.di-pietro@umontpellier.fr}}
\author{Simon Lemaire}
\address{Inria, Univ.~Lille, CNRS, UMR 8524 -- Laboratoire Paul Painlev\'e\\ F-59000 Lille, France\\\email{simon.lemaire@inria.fr}}

\else

\usepackage[hmargin={28mm,28mm},vmargin={30mm,35mm}]{geometry}
\usepackage{amsthm}
\usepackage{authblk}
\usepackage{newtxtext,newtxmath}
\newtheorem{theorem}{Theorem}
\newtheorem{lemma}[theorem]{Lemma}
\newtheorem{corollary}[theorem]{Corollary}
\newtheorem{remark}[theorem]{Remark}
\newcommand{\email}[1]{\href{mailto:#1}{#1}}

\title{A discrete Weber inequality on three-dimensional hybrid spaces with application to the HHO approximation of magnetostatics}
\author[1]{Florent Chave\footnote{\email{florent.chave@inria.fr}}}
\affil[1]{Inria, Univ.~Lille, CNRS, UMR 8524 -- Laboratoire Paul Painlev\'e, F-59000 Lille, France}
\author[2]{Daniele A.~Di Pietro\footnote{\email{daniele.di-pietro@umontpellier.fr}}}
\affil[2]{IMAG, Univ.~Montpellier, CNRS, Montpellier, France}
\author[1]{Simon Lemaire\footnote{\email{simon.lemaire@inria.fr} (corresponding author)}}

\fi

\newcommand{\Real}{\mathbb{R}}
\newcommand{\Natural}{\mathbb{N}}
\def\normal{{\boldsymbol{n}}}
\def\f{{\mathcal{F}}}
\def\t{{\mathcal{T}}}
\newcommand{\st}{~:~}
\DeclareMathOperator{\Div}{div}
\DeclareMathOperator{\Curl}{\bf curl}
\DeclareMathOperator{\Grad}{\bf grad}
\newcommand{\Hdiv}[1][\Omega]{\boldsymbol{H}(\Div ; #1)}

\newcommand{\Hdivz}[1][\Omega]{\boldsymbol{H}(\Div^0 ; #1)}

\newcommand{\Hcurl}[1][\Omega]{\boldsymbol{H}(\Curl ; #1)}
\newcommand{\Hzcurl}[1][\Omega]{\boldsymbol{H}_{\boldsymbol{0}}(\Curl ; #1)}
\newcommand{\Poly}{\mathcal{P}}

\def\PkdTh{ {\boldsymbol\Poly^k(\t_h)} }
\def\PkpdT{ {\boldsymbol\Poly^{k+1}(T)} }
\def\PkpdTh{ {\boldsymbol\Poly^{k+1}(\t_h)} }
\def\PkT{ {\Poly^k(T)} }
\def\PkTh{ {\Poly^k(\t_h)} }

\def\PkppT{ {\Poly^{k+2}(T)} }

\def\PkdF{ {\boldsymbol\Poly^k(F)} }

\def\PkpF{ {\Poly^{k+1}(F)} }

\def\XT{ {\underline{\boldsymbol{\rm X}}}_T^{k+1} }
\def\XTs{ {\underline{\boldsymbol{\rm X}}}_{\flat,T}^{k+1} }
\def\Xh{ {\underline{\boldsymbol{\rm X}}}_h^{k+1} }
\def\Xhs{ {\underline{\boldsymbol{\rm X}}}_{\flat,h}^{k+1} }
\def\Xhz{ {\underline{\boldsymbol{\rm X}}_{h,\boldsymbol{0}}^{k+1}} }
\def\Xhzs{ {\underline{\boldsymbol{\rm X}}_{\flat,h,\boldsymbol{0}}^{k+1}} }
\def\XXhz{ \overset{\circ}{{\underline{\boldsymbol{\rm X}}}}{}_{h,\boldsymbol{0}}^{k+1} }
\def\XXhzs{ \overset{\circ}{{\underline{\boldsymbol{\rm X}}}}{}_{\flat,h,\boldsymbol{0}}^{k+1} }

\def\YT{ {\underline{\rm Y}}_T^{k+1} }
\def\Yh{ {\underline{\rm Y}}_h^{k+1} }
\def\Yhz{ {\underline{\rm Y}_{h,0}^{k+1}} }

\def\Zh{ {\underline{\mathbb{Z}}_h^{k+1}} }
\def\Zhs{ {\underline{\mathbb{Z}}_{\flat,h}^{k+1}} }
\def\Zhz{ {\underline{\mathbb{Z}}_{h,\mathbbl{0}}^{k+1}} }
\def\Zhzs{ {\underline{\mathbb{Z}}_{\flat,h,\mathbbl{0}}^{k+1}} }
\def\ZZhz{ \overset{\circ}{{\underline{\mathbb{Z}}}}{}_{h,\mathbbl{0}}^{k+1} }
\def\ZZhzs{ \overset{\circ}{{\underline{\mathbb{Z}}}}{}_{\flat,h,\mathbbl{0}}^{k+1} }
\def\IntXh{ {\underline{\boldsymbol{\rm I}}_{\boldsymbol{{\rm X}},h}^{k+1}} }
\def\IntXhs{ {\underline{\boldsymbol{\rm I}}_{\boldsymbol{{\rm X}},\flat,h}^{k+1}} }
\def\IntYh{ {\underline{\rm I}_{{\rm Y},h}^{k+1}} }
\def\IntUh{ {\underline{\boldsymbol{\rm I}}_{\boldsymbol{{\rm X}},h}^{k+1}} }
\def\IntPh{ {\underline{\rm I}_{{\rm Y},h}^{k+1}} }
\def\Gkh{ {\boldsymbol{G}_h^{k+1}} }
\def\GkT{ {\boldsymbol{G}_T^{k+1}} }
\def\Ckh{ {\boldsymbol{C}_h^{k}} }
\def\CkT{ {\boldsymbol{C}_T^{k}} }
\def\hatuh{{ \widehat{\underline{\boldsymbol{\rm u}}}_h }} 
\def\erruh{{ \underline{\boldsymbol{\rm e}}_h }} 
\def\uh{{ \underline{\boldsymbol{\rm u}}_h }} 
\def\vh{{ \underline{\boldsymbol{\rm v}}_h }}
\def\vhs{{ \underline{\boldsymbol{\rm v}}_h^\star }}
\def\wh{{ \underline{\boldsymbol{\rm w}}_h }}

\def\zh{{ \underline{\mathbbl{z}}_h }}
\def\hatph{{ \widehat{\underline{\rm p}}_h }} 
\def\errph{{ \underline{\upepsilon}_h }} 
\def\ph{{ \underline{\rm p}_h }} 
\def\qh{{ \underline{\rm q}_h }} 
\def\rh{{ \underline{\rm r}_h }} 
\def\hatuTF{{ \widehat{\underline{\boldsymbol{\rm u}}}_T }}

\def\vTF{{ \underline{\boldsymbol{\rm v}}_T }}

\def\wTF{{ \underline{\boldsymbol{\rm w}}_T }}

\def\pTF{{ \underline{\rm p}_T }} 
\def\qTF{{ \underline{\rm q}_T }} 
\def\rTF{{ \underline{\rm r}_T }} 

\def\uTh{{ {\boldsymbol{\rm u}}_h }} 
\def\vTh{{ {\boldsymbol{\rm v}}_h }}
\def\vThs{{ {\boldsymbol{\rm v}}_h^\star }}
\def\wTh{{ {\boldsymbol{\rm w}}_h }}

\def\qTh{{ {\rm q}_h }} 
\def\rTh{{ {\rm r}_h }} 

\def\vT{{ {\boldsymbol{\rm v}}_T }}
\def\vTs{{ {\boldsymbol{\rm v}}_T^\star }}
\def\wT{{ {\boldsymbol{\rm w}}_T }}

\def\qT{{ {\rm q}_T }}

\def\vF{{ {\boldsymbol{\rm v}}_F }}
\def\vFs{{ {\boldsymbol{\rm v}}_F^\star }}
\def\wF{{ {\boldsymbol{\rm w}}_F }}

\def\qF{{ {\rm q}_F }} 
\def\rF{{ {\rm r}_F }} 
\def\Ah{{ \mathrm{A}_h }} 
\def\ah{{ {\rm a}_h }} 
\def\bh{{ {\rm b}_h }} 
\def\ch{{ {\rm c}_h }}
\def\dh{{ {\rm d}_h }}

\def\lh{{ {\rm l}_h }} 
\def\mh{{ {\rm m}_h }} 
\def\sh{{ {\rm s}_h }}

\newcommand{\defi}{\mathrel{\mathop:}=}
\newcommand{\ifed}{=\mathrel{\mathop:}}
\renewcommand{\vec}[1]{\boldsymbol{#1}}

\newcommand{\logLogSlopeTriangle}[5]
{
    \pgfplotsextra
    {
        \pgfkeysgetvalue{/pgfplots/xmin}{\xmin}
        \pgfkeysgetvalue{/pgfplots/xmax}{\xmax}
        \pgfkeysgetvalue{/pgfplots/ymin}{\ymin}
        \pgfkeysgetvalue{/pgfplots/ymax}{\ymax}

        \pgfmathsetmacro{\xArel}{#1}
        \pgfmathsetmacro{\yArel}{#3}
        \pgfmathsetmacro{\xBrel}{#1-#2}
        \pgfmathsetmacro{\yBrel}{\yArel}
        \pgfmathsetmacro{\xCrel}{\xArel}

        \pgfmathsetmacro{\lnxB}{\xmin*(1-(#1-#2))+\xmax*(#1-#2)} 
        \pgfmathsetmacro{\lnxA}{\xmin*(1-#1)+\xmax*#1} 
        \pgfmathsetmacro{\lnyA}{\ymin*(1-#3)+\ymax*#3} 
        \pgfmathsetmacro{\lnyC}{\lnyA+#4*(\lnxA-\lnxB)}
        \pgfmathsetmacro{\yCrel}{\lnyC-\ymin)/(\ymax-\ymin)}

        \coordinate (A) at (rel axis cs:\xArel,\yArel);
        \coordinate (B) at (rel axis cs:\xBrel,\yBrel);
        \coordinate (C) at (rel axis cs:\xCrel,\yCrel);

        \draw[#5]   (A)-- node[pos=0.5,anchor=north] {\scriptsize{1}}
                    (B)-- 
                    (C)-- node[pos=0.,anchor=west] {\scriptsize{#4}} 
                    (A);
    }
}

\newcommand{\logLogSlopeTriangleNDOFs}[5]
{
    \pgfplotsextra
    {
        \pgfkeysgetvalue{/pgfplots/xmin}{\xmin}
        \pgfkeysgetvalue{/pgfplots/xmax}{\xmax}
        \pgfkeysgetvalue{/pgfplots/ymin}{\ymin}
        \pgfkeysgetvalue{/pgfplots/ymax}{\ymax}

        \pgfmathsetmacro{\xArel}{#1}
        \pgfmathsetmacro{\yArel}{#3}
        \pgfmathsetmacro{\xBrel}{#1-#2}
        \pgfmathsetmacro{\yBrel}{\yArel}
        \pgfmathsetmacro{\xCrel}{\xArel}

        \pgfmathsetmacro{\lnxB}{\xmin*(1-(#1-#2))+\xmax*(#1-#2)} 
        \pgfmathsetmacro{\lnxA}{\xmin*(1-#1)+\xmax*#1} 
        \pgfmathsetmacro{\lnyA}{\ymin*(1-#3)+\ymax*#3} 
        \pgfmathsetmacro{\lnyC}{\lnyA-#4*(\lnxA-\lnxB)}
        \pgfmathsetmacro{\yCrel}{\lnyC-\ymin)/(\ymax-\ymin)}

        \coordinate (A) at (rel axis cs:\xArel,\yArel);
        \coordinate (B) at (rel axis cs:\xBrel,\yBrel);
        \coordinate (C) at (rel axis cs:\xCrel,\yCrel);

        \draw[#5]   (A)-- node[pos=0.5,anchor=north] {\scriptsize{1}}
                    (B)-- 
                    (C)-- node[pos=0.,anchor=east] {\scriptsize{#4}} 
                    (A);
    }
}

\newcounter{corr}
\definecolor{violet}{rgb}{0.580,0.,0.827}

\usepackage{changepage}
\usepackage{marginnote}

\begin{document}

\maketitle

\ifMMMAS
\else

\begin{abstract}
  We prove a discrete version of the first Weber inequality on three-dimensional hybrid spaces spanned by vectors of polynomials attached to the elements and faces of a polyhedral mesh. We then introduce two Hybrid High-Order methods for the approximation of the magnetostatics model, in both its (first-order) field and (second-order) vector potential formulations. These methods are applicable on general polyhedral meshes, and allow for arbitrary orders of approximation. Leveraging the previously established discrete Weber inequality, we perform a comprehensive analysis of the two methods. We finally validate them on a set of test-cases.
\end{abstract}

\fi

\section{Introduction}

Let $\Omega\subset\Real^3$ denote an open, bounded, and connected polyhedral domain.
In the study of problems in electromagnetism, Weber inequalities~\cite{Weber:80} constitute a very powerful tool. They can be viewed as a generalization of the celebrated Poincar\'e inequality to the case of vector fields belonging to $\Hcurl\cap\Hdiv$, and featuring either vanishing tangential component (first Weber inequality), or vanishing normal component (second Weber inequality) on the boundary $\partial\Omega$ of the domain. We refer the reader to~\cite[Theorems 3.4.3 and 3.5.3]{Assous.ea:18} for a general (from a topological viewpoint) statement of Weber inequalities.

We next state the (continuous) first Weber inequality.
For the sake of simplicity, we assume from now on that $\Omega$ is simply-connected and that $\partial\Omega$ is connected.
Under these assumptions, the first and second Betti numbers of $\Omega$ are both zero, i.e., $\Omega$ does not have tunnels and does not enclose any void.
For a deeper insight into the role of the different topological assumptions we make on the domain, we refer to Remark~\ref{rem:top}.
For any measurable set $X\subset\overline{\Omega}$, we irrespectively denote by $(\cdot,\cdot)_X$ and $||\cdot||_X$ the usual inner products and norms on the scalar-valued space $L^2(X)$ and on the vector-valued spaces $\boldsymbol{L}^2(X;\mathbb{R}^\ell)$ for $\ell\in\{2,3\}$.
We also set, letting $\normal$ denote the unit normal vector field on $\partial\Omega$ pointing out of $\Omega$, $\Hzcurl\defi\left\{\vec{v}\in\Hcurl\st\vec{n}{\times}(\vec{v}{\times}\vec{n})=\vec{0}\text{ on }\partial\Omega\right\}$, and
\begin{alignat*}{1}
  \Hdivz\defi&\left\{\vec{v}\in\Hdiv\st\Div\vec{v}=0\text{ in }\Omega\right\}\\=&\left\{\vec{v}\in \boldsymbol L^2(\Omega;\mathbb{R}^3)\st
  (\vec{v},\Grad\varphi)_\Omega=0\quad\forall\varphi\in H^1_0(\Omega)
  \right\}.
\end{alignat*}
The following $\boldsymbol L^2(\Omega;\mathbb{R}^3)$-orthogonal decomposition holds true (cf.~\cite[Proposition 3.7.2]{Assous.ea:18}):
\begin{equation} \label{eq:helmholtz}
  \Hzcurl={\bf grad}\big(H^1_0(\Omega)\big)\overset{\perp}{\oplus}\left(\Hzcurl\cap\Hdivz\right).
\end{equation}
With the assumptions we have made on the topology of the domain $\Omega$, the first Weber inequality reads: For any $\vec{v}\in\Hzcurl\cap\Hdivz$,
\begin{equation} \label{eq:weber}
  \|\vec{v}\|_\Omega\leq C_{\rm W}\,\|\Curl\vec{v}\|_\Omega,
\end{equation}
for some constant $C_{\rm W}>0$ only depending on the domain $\Omega$.
In this work, we derive a discrete version of the first Weber inequality~\eqref{eq:weber} on (three-dimensional) hybrid spaces spanned by vectors of polynomials attached to the elements and faces of a (polyhedral) mesh, as they can be encountered in Hybridizable Discontinuous Galerkin~\cite{Cockburn.ea:09} and related~\cite{Wang.Ye:13} methods, or in Hybrid High-Order (HHO)~\cite{DP.Ern.Lemaire:14,DP.Ern:15} methods; see~\ifMMMAS\else\cite{Cockburn.DP.Ern:16} \fi for a discussion highlighting the analogies and differences between these two families of methods in the context of scalar variable diffusion. The corresponding result is stated in Theorem~\ref{thm:max.ineq}. The proof extends the general ideas used in~\cite[Lemma~2.15]{Di-Pietro.Droniou:20} to derive a discrete Poincar\'e inequality on hybrid spaces.
A related (yet weaker) result for (non-hybrid) broken polynomial spaces on tetrahedral meshes has been derived in~\cite[Lemma~4.1]{Chen.ea:18}.
Poincar\'e--Friedrichs inequalities for complexes of discrete distributional differential forms (introduced in~\ifMMMAS\else\cite{Licht:17} \fi as a generalization of the distributional finite element de Rham sequences of~\cite[Section 3]{Braess.Schoberl:08}) are also proved in the recent work~\ifMMMAS\else\cite{Christiansen.Licht:20}\fi.
  Complexes of discrete distributional differential forms express a notion of compatibility that is relieved from global conformity requirements through the presence of ``jump terms'' in the distributions (with conforming finite elements corresponding to the special case where such terms vanish).
A major difference with respect to the present contribution is that we work in a non-compatible setting, i.e., there is no underlying exact discrete complex. Specifically, the sole result related to (discrete) compatibility leveraged in the proof of the discrete Weber inequality is a decomposition of vector-valued polynomial functions inside each mesh element (see Lemma \ref{lemma:ineq.decomp} below). Another important difference between the two works is that our results naturally apply to general polyhedral meshes.

In the second part of this paper, we tackle the HHO approximation of magnetostatics, in both its (first-order) field formulation and (second-order, generalized) vector potential formulation. Various discretization methods have been studied in the literature to approximate the magnetostatics equations (or, more generally, Maxwell equations). 
Conforming finite element discretizations were originally proposed (on tetrahedra, essentially) in the seminal work of N\'ed\'elec~\cite{Nedelec:80,Nedelec:86} (cf.~also~\ifMMMAS\else\cite{Monk:03}\fi). We also mention~\ifMMMAS\else\cite{Arnold:18}\fi, in which a unified presentation of conforming finite element methods based on notions from algebraic topology is provided.
Nonconforming discretizations include, on standard meshes, the Discontinuous Galerkin method of~\ifMMMAS\else\cite{Perugia.ea:02} \fi as well as the Hybridizable Discontinuous Galerkin method of~\ifMMMAS\else\cite{Nguyen.ea:11,Chen.Cui.Xu:19} \fi and, on general polyhedral meshes, the variant~\ifMMMAS\else\cite{Mu.Wang.ea:15} \fi of~\ifMMMAS\else\cite{Nguyen.ea:11} \fi and the method of~\ifMMMAS\else\cite{Chen.ea:17}\fi; see also~\ifMMMAS\else\cite{Du.ea:20}\fi.
Methods that support general polyhedral meshes and are built upon discrete spaces that mimick the continuity properties of the spaces appearing in the continuous weak formulation include the Virtual Element methods of~\ifMMMAS\else\cite{Beirao.ea:18a,Beirao.ea:18b,Beirao.ea:18c}\fi, and the fully discrete method of~\ifMMMAS\else\cite{Di-Pietro.Droniou:20*1} \fi based on the discrete de Rham sequence of~\ifMMMAS\else\cite{Di-Pietro.Droniou.ea:20} \fi (see also~\ifMMMAS\else\cite{Di-Pietro.Droniou:21} \fi for recent developments including error estimates).
All the hybridized nonconforming methods cited above deal with the approximation of magnetostatics under its (generalized) vector potential formulation. In this paper, we first study an HHO method (which has been briefly introduced in~\ifMMMAS\else\cite{Chave.ea:20}\fi) for magnetostatics under its field formulation. We take advantage of the fact that the corresponding problem is first-order to avoid locally reconstructing a discrete ${\bf curl}$ operator as it has to be done for second-order problems (cf.~Remark~\ref{rem:curl}). Doing so, we propose a computationally inexpensive and easy-to-implement method. Second, we study an HHO method for magnetostatics under its (generalized) vector potential formulation, that can be seen as a computationally cheaper variant of the method introduced in~\ifMMMAS\else\cite{Chen.ea:17} \fi (cf.~Remark~\ref{rem:variant}). Our two HHO methods are applicable on general polyhedral meshes, and allow for an arbitrary order of approximation $k\geq 0$ with proved energy-error of order $k+1$ (cf.~Theorems~\ref{thm:standard.error} and~\ref{thm:error}). Leveraging the previously established discrete Weber inequality, we carry out a comprehensive analysis of the methods, and validate them on a set of test-cases.

The article is organized as follows. In Section~\ref{se:weber}, we prove the discrete Weber inequality.
Then, in Sections~\ref{sse:field} and~\ref{sec:general}, we tackle the HHO approximation of magnetostatics under its field and vector potential formulations, respectively.


\section{A discrete Weber inequality on hybrid spaces} \label{se:weber}

\subsection{Discrete setting} \label{se:disset}

We consider a polyhedral mesh $\mathcal{M}_h=(\t_h,\f_h)$, that is assumed to belong to a regular mesh sequence in the sense of~\cite[Definition 1.9]{Di-Pietro.Droniou:20}.
The set $\t_h$ is a finite collection of nonempty, disjoint, open polyhedra $T$ (called mesh elements) such that $\overline{\Omega} = \bigcup_{T\in\t_h} \overline{T}$. For convenience reasons that will be made clear in Remark~\ref{rem:role.star-shaped} below, we henceforth assume that any mesh element $T\in\t_h$ is star-shaped with respect to some interior point $\vec{x}_T$. Note that this assumption is not necessary for the results of this article to hold true.
The subscript $h$ refers to the meshsize, defined by $h \defi \max_{T\in\t_h} h_T$, where $h_T$ denotes the diameter of the element $T$. The set $\f_h$ collects the planar mesh faces and, for all $T\in\t_h$, we denote by $\f_T$ the set of faces that lie on the boundary of $T$.
Boundary faces lying on $\partial\Omega$ are collected in the set $\f_h^{\text{b}}$, and we denote by $\f_h^{\text{i}}\coloneq\f_h\setminus\f_h^{\text{b}}$ the set of interfaces.
For all $F\in\f_h$, we let $h_F$ denote its diameter and, for all $T\in\t_h$ and all $F\in\f_T$, $\normal_{TF}$ denote the unit normal vector to $F$ pointing out of $T$.
We recall that, since ${\cal M}_h$ belongs to a regular mesh sequence, for all $T\in\t_h$, the quantity ${\rm card}(\f_T)$ is bounded from above uniformly in $h$ and, for all $T\in\t_h$ and all $F\in\f_T$, $h_F$ is uniformly comparable to $h_T$ (cf.~\cite[Lemma 1.12]{Di-Pietro.Droniou:20}).
In what follows, we will use the symbol $\lesssim$ to indicate that an estimate is valid up to a multiplicative constant that may depend on the mesh regularity parameter, the ambient dimension, and (if need be) the polynomial degree, but that is independent of $h$.

\subsection{Hybrid spaces} \label{sse:hybrid}

For $q\in\Natural$ and $(X,\ell)\in\{(\f_h,\{2\}),(\t_h,\{3\})\}$, we respectively denote by $\Poly^q(X)$ and $\boldsymbol\Poly^q(X)$ the vector spaces of $\ell$-variate, scalar-valued and $\mathbb{R}^\ell$-valued polynomial functions on $X$ of total degree at most $q$.
For future use, for any $T\in\t_h$, we let $\boldsymbol{{\cal G}}^q(T)\defi{\bf grad}\big(\Poly^{q+1}(T)\big)$ and $\boldsymbol{{\cal R}}^q(T)\defi{\bf curl}\big(\boldsymbol\Poly^{q+1}(T)\big)$ (the notation $\boldsymbol{{\cal R}}$ standing for ``rot''), and we recall that the following (nonorthogonal) decomposition holds true:
\begin{equation} \label{eq:decomp}
  \boldsymbol\Poly^q(T)=\boldsymbol{{\cal G}}^q(T)\oplus(\vec{x}-\vec{x}_T){\times}\,\boldsymbol{{\cal R}}^{q-1}(T),
\end{equation}
with the convention that $\boldsymbol{{\cal R}}^{-1}(T)\defi\{\vec{0}\}$.
For any $F\in\f_h$, we also let
$$\boldsymbol{{\cal G}}^{q}(F)\defi{\bf grad}_{\tau}\big(\Poly^{q+1}(F)\big)\subseteq\boldsymbol\Poly^{q}(F)$$
denote the space of (tangential) gradients of polynomials of total degree up to $q+1$ on $F$.
Finally, we define the broken spaces
\begin{align*}
  \Poly^q(\t_h)&\defi\left\{v\in L^2(\Omega)\st v_{\mid T}\in\Poly^q(T)\quad\forall T\in\t_h\right\},
  \\
  \boldsymbol\Poly^q(\t_h)&\defi\left\{\vec{v}\in \boldsymbol L^2(\Omega;\mathbb{R}^3)\st\vec{v}_{\mid T}\in\boldsymbol\Poly^q(T)\quad\forall T\in\t_h\right\},
\end{align*}
as well as the broken subspaces $\boldsymbol{{\cal G}}^q({\t_h})\defi{\bf grad}_h\big(\Poly^{q+1}(\t_h)\big)$ and $\boldsymbol{{\cal R}}^q({\t_h})\defi{\bf curl}_h\big(\boldsymbol\Poly^{q+1}(\t_h)\big)$, where $\Grad_h$ (resp.~$\Curl_h$) denotes the usual broken $\Grad$ (resp.~$\Curl$) operator on $H^1(\t_h)$ (resp.~$\Hcurl[\t_h]$).

Let an integer polynomial degree $k\ge 0$ be given. We introduce the following (global) hybrid spaces:
\begin{subequations}\label{def:DOFs}
  \begin{alignat}{1}\label{def:DOFs.Xh}
  	\Xh &\defi\left\{ \vh = \big((\vT)_{T\in\t_h},(\vF)_{F\in\f_h}\big) \st 
  	  \begin{alignedat}{2}
  		\vT&\in \PkpdT &\quad& \forall T\in\t_h
  		\\
  		\vF&\in\boldsymbol{{\cal G}}^{k+1}(F) &\quad& \forall F\in\f_h
  	  \end{alignedat}
  	\right\},
  	\\\label{def:DOFs.Yh}
  	\Yh & \defi \left\{ \qh = \big((\qT)_{T\in\t_h},(\qF)_{F\in\f_h}\big)\st 
  	  \begin{alignedat}{2}
  	  	\qT&\in\PkT &\quad& \forall T\in\t_h
  	  	\\
  	  	\qF&\in\PkpF &\quad& \forall F\in\f_h
  	  \end{alignedat}
  	\right\},
  \end{alignat}
\end{subequations}
as well as their subspaces incorporating homogeneous essential boundary conditions:
\begin{equation}\label{def:DOFs.Xhz}
	\begin{aligned}
		\Xhz &\defi \left\{\vh\in\Xh\st
			\vF = \boldsymbol{0} \quad\forall F\in\f_h^{\rm b}		
		\right\},
		\\
		\Yhz &\defi \left\{\qh\in\Yh\st
			\qF = 0 \quad\forall F\in\f_h^{\rm b}		
		\right\}.
	\end{aligned}
\end{equation}
(Semi)norms on the above spaces are defined in \eqref{def:norms.h:curl} and \eqref{eq:norm.Y.h} (or~\eqref{def:norms.general:grad}) below.
In~\eqref{def:DOFs.Xh}, the face unknowns for the vectorial variable lie in a strict subspace of $\boldsymbol\Poly^{k+1}(F)$, that is $\boldsymbol{{\cal G}}^{k+1}(F)$. This choice is driven by stability purposes. Since the ${\bf curl}$ operator has a kernel on $\boldsymbol\Poly^{k+1}(T)$ that is composed of all polynomials in $\boldsymbol{{\cal G}}^{k+1}(T)$ (this will be justified by Lemma~\ref{lemma:ineq.decomp} below), we must be able to control the (jump of the) tangential traces of all functions in that space at interfaces, which is the reason why the face unknowns must (at least) lie in $\boldsymbol{{\cal G}}^{k+1}(F)$. As far as~\eqref{def:DOFs.Yh} is concerned, the necessity to consider face unknowns in $\Poly^{k+1}(F)$ for the scalar variable is driven by the necessity to reconstruct a ${\bf grad}$ operator in $\PkpdT$ (in turn driven by the necessity to test this discrete ${\bf grad}$ operator against element unknowns for the vectorial variable themselves in $\PkpdT$); cf.~Section~\ref{sec:grad.recons} below.
Given a mesh element $T\in\t_h$, we respectively denote by $\XT$ and $\YT$ the restrictions of $\Xh$ and $\Yh$ to $T$, and by $\vTF\defi\big(\vT,(\vF)_{F\in\f_T}\big)\in\XT$ and $\qTF\defi\big(\qT,(\qF)_{F\in\f_T}\big)\in\YT$ the respective restrictions of generic vectors of polynomials $\vh\in\Xh$ and $\qh\in\Yh$. 
Also, we let $\vTh$ and $\qTh$ (not underlined) be the broken polynomial functions in $\PkpdTh$ and in $\PkTh$ such that
$$(\vTh)_{\mid T} \defi \vT \quad\text{ and }\quad (\qTh)_{\mid T} \defi \qT \quad\text{ for all } T\in\t_h.$$
We next define the interpolators on $\Xh$ and $\Yh$.
To this purpose, we need some preliminary definitions.
  For any open set $X\subseteq\Omega$ and any $F\in\f_h$ such that $F\subset\overline{X}$, $\vec{\gamma}_{\tau,F}(\boldsymbol{v})\in \boldsymbol L^2(F;\mathbb{R}^2)$ denotes the tangential trace on $F$ of $\boldsymbol{v}\in \boldsymbol H^1(X;\mathbb{R}^3)$.
  For $q\in\mathbb{N}$ and $(X,\ell)\in\{(\f_h,\{2\}),(\t_h,\{3\})\}$, we respectively denote by $\pi_{{\cal P},X}^{q}$ and $\vec{\pi}_{\boldsymbol{\cal P}, X}^q$ the $L^2(X)$-orthogonal projector onto $\Poly^q(X)$ and the $\boldsymbol L^2(X;\mathbb{R}^\ell)$-orthogonal projector onto~$\boldsymbol \Poly^q(X)$.
  Similarly, $\vec{\pi}_{\boldsymbol{{\cal G}},F}^{q}$ stands for the $\boldsymbol L^2(F;\mathbb{R}^2)$-orthogonal projector onto $\boldsymbol{{\cal G}}^{q}(F)$ and, for $\boldsymbol{{\cal S}}\in\{\boldsymbol{{\cal G}},\boldsymbol{{\cal R}}\}$, $\vec{\pi}_{\boldsymbol{{\cal S}},T}^q$ for the $\boldsymbol L^2(T;\mathbb{R}^3)$-orthogonal projector onto $\boldsymbol{{\cal S}}^q(T)$.
  We finally let $\IntXh: \boldsymbol H^1(\Omega;\mathbb{R}^3)\rightarrow\Xh$ and $\IntYh: H^1(\Omega)\rightarrow\Yh$ be such that, for all $\boldsymbol{v}\in \boldsymbol H^1(\Omega;\mathbb{R}^3)$ and all $q\in H^1(\Omega)$,
\begin{subequations}
  \begin{alignat}{1}\label{def:Icurl.h}
    \IntXh \boldsymbol{v} &\defi \left(
    \big({\vec{\pi}}_{\boldsymbol{\cal P},T}^{k+1}(\boldsymbol{v}_{\mid T})\big)_{T\in\t_h},
    \big({\vec{\pi}}_{\boldsymbol{{\cal G}},F}^{k+1}\big(\vec{\gamma}_{\tau,F}(\boldsymbol{v})\big)\big)_{F\in\f_h}
    \right),
	\\\label{def:Igrad.h}
	\IntYh q &\defi \left(\big({\pi}_{{\cal P},T}^{k}(q_{\mid T})\big)_{T\in\t_h},\big({\pi}_{{\cal P},F}^{k+1}(q_{\mid F})\big)_{F\in\f_h}\right).
  \end{alignat}
\end{subequations}
For future use, we close this section by introducing the global $L^2(\Omega)$-orthogonal projector $\pi_{{\cal P},h}^{q}$, and $\boldsymbol{L}^2(\Omega;\mathbb{R}^3)$-orthogonal projectors $\vec{\pi}_{\boldsymbol{\cal P},h}^q$ and $\vec{\pi}_{\boldsymbol{{\cal S}},h}^q$ onto, respectively, $\Poly^q(\t_h)$, $\boldsymbol\Poly^q(\t_h)$ and $\boldsymbol{{\cal S}}^q({\t_h})$.

\subsection{Gradient reconstruction in $\Yh$}\label{sec:grad.recons}

We define the global discrete gradient reconstruction operator $\Gkh:\Yh\rightarrow \PkpdTh$ such that its local restriction $\GkT:\YT\rightarrow \PkpdT$ to any $T\in\t_h$ solves the following problem: For all $\qTF\in\YT$,
\begin{equation}\label{def:GTk}
	\big(\GkT \qTF,\vec{w}\big)_T = -(\qT,\Div\vec{w})_T + \sum_{F\in\f_T}(\qF,\vec{w}_{\mid F}{\cdot}\normal_{TF})_F\quad\forall\vec{w}\in\PkpdT.
\end{equation}
By the Riesz representation theorem in $\PkpdT$ equipped with the $\boldsymbol L^2(T;\mathbb{R}^3)$-inner product, $\GkT \qTF$ is uniquely defined.
We note the following commutation property (see, e.g., \cite[Section 4.2.1]{Di-Pietro.Droniou:20}):
For all $q\in H^1(\Omega)$, we have
\begin{equation}\label{lem:commut.GTkITk.eq}
  \Gkh\big(\IntYh q\big) = \vec{\pi}_{\boldsymbol{\cal P},h}^{k+1} (\Grad q).
\end{equation}
We also have the following result in the tetrahedral case, which can be exploited to simplify the HHO schemes of Sections \ref{sse:field} and \ref{sec:general} on matching tetrahedral meshes; see Remarks \ref{rk:tet.case.1} and \ref{rk:tet.case} below.
\begin{lemma}[Norm $\|\Gkh\cdot\|_{\Omega}$]\label{le:GTk.norm}
  Let $\t_h$ be a matching tetrahedral mesh. Then, the map $\|\Gkh\cdot\|_{\Omega}$ defines a norm on $\Yhz$.
\end{lemma}
\begin{proof}
  Let $\t_h$ be a matching tetrahedral mesh, and $\qh\in\Yhz$ be such that $\|\Gkh\qh\|_{\Omega}=0$.
  Then, for all $T\in\t_h$, enforcing $\GkT\qTF=\boldsymbol{0}$ in the definition~\eqref{def:GTk} of this quantity and integrating by parts yields
  \begin{equation} \label{eq:pre.rob}
    \big(\Grad\qT,\vec{w}\big)_T+\sum_{F\in\f_T}\big(\qF-{\rm q}_{T\mid F},\vec{w}_{\mid F}{\cdot}\normal_{TF}\big)_F=0\qquad\forall\vec{w}\in\boldsymbol\Poly^{k+1}(T).
  \end{equation}
  Let $\boldsymbol{{\cal N}}^k(T)\defi\boldsymbol{{\cal G}}^{k-1}(T)\oplus(\vec{x}-\vec{x}_T){\times}\,\boldsymbol{{\cal R}}^{k-1}(T)$ (with the convention that $\boldsymbol{{\cal G}}^{-1}(T)\defi\{\boldsymbol{0}\}$) denote the N\'ed\'elec space of the first kind of degree $k$ on $T$ (cf.~\ifMMMAS\else\cite{Nedelec:80}\fi), and let $\overline{\vec{w}}\in\boldsymbol\Poly^{k+1}(T)$ be such that
  \begin{subequations} \label{eq:bdm}
    \begin{alignat}{2}
      (\overline{\vec{w}},\vec{p})_T&=(\Grad\qT,\vec{p})_T\qquad&\forall\vec{p}\in\boldsymbol{{\cal N}}^k(T),\\
      (\overline{\vec{w}}_{\mid F}{\cdot}\normal_{TF},r)_F&=(\qF-{\rm q}_{T\mid F},r)_F\qquad&\forall r\in\Poly^{k+1}(F),\quad\forall F\in\f_T.
    \end{alignat}
  \end{subequations}
  The system~\eqref{eq:bdm} uniquely defines $\overline{\vec{w}}$ as a function of the N\'ed\'elec space of the second kind of degree $k+1$ on $T$, that is $\boldsymbol\Poly^{k+1}(T)$ (cf.~\ifMMMAS\else\cite{Nedelec:86,Brezzi.ea:85}\fi).
  Testing~\eqref{eq:pre.rob} with $\overline{\vec{w}}$, and using that $\Grad\qT\in\boldsymbol{{\cal G}}^{k-1}(T)\subseteq\boldsymbol{{\cal N}}^k(T)$ and that $\qF-{\rm q}_{T\mid F}\in\Poly^{k+1}(F)$ for all $F\in\f_T$, we infer from~\eqref{eq:bdm} that $\|\Grad\qT\|_T^2+\sum_{F\in\f_T}\|\qF-{\rm q}_{T\mid F}\|_F^2=0$.
  Reproducing the same reasoning on all $T\in\t_h$, and using that $\qh$ belongs to the space $\Yhz$ with strongly enforced boundary conditions, finally yields $\qh=\underline{0}_h$.
\end{proof}

\subsection{Discrete Weber inequality} \label{sse:dwi}

Let us begin with a preliminary technical result.

\begin{lemma}[Polynomial decomposition]\label{lemma:ineq.decomp}
  Let $q\in\mathbb{N}$.
  Let $T\in\t_h$ and $\vec{p}\in\boldsymbol\Poly^q(T)$.
    Then, there exist $g\in\Poly^{q+1}(T)$ and $\vec{c}\in\boldsymbol\Poly^q(T)$ such that $\vec{p} = \Grad g + {(\boldsymbol{x}-\boldsymbol{x}_T)}\times\Curl \vec{c}$ and
	\begin{equation}\label{lemma:ineq.decomp.eq}
		\|\vec{p}-\Grad g\|_{T} \leq 2 h_T\|\Curl \vec{p}\|_{T}.
	\end{equation}
\end{lemma}
\begin{proof}
  The decomposition $\vec{p} = \Grad g + {(\boldsymbol{x}-\boldsymbol{x}_T)}\times\Curl \vec{c}$ directly follows from~\eqref{eq:decomp}.
  Using the fact that ${|(\boldsymbol{x}-\boldsymbol{x}_T){\times}\Curl\vec{c}|}\leq{|\boldsymbol{x}-\boldsymbol{x}_T|}~{|\Curl\vec{c}|}\leq h_T|\Curl\vec{c}|$, we infer
  \begin{equation}\label{proof.ineq.decomp.0}
    \|\vec{p}-\Grad g\|_T = \|(\boldsymbol{x}-\boldsymbol{x}_T){\times}\Curl \vec{c}\|_T \leq h_T\|\Curl \vec{c}\|_T.
  \end{equation}
  We now focus on the right-hand side of \eqref{proof.ineq.decomp.0}.
  Since $\boldsymbol{\rm curl}\big(\Grad g\big)= \boldsymbol{0}$, we have
  \[
  \Curl \vec{p} = \Curl \big((\boldsymbol{x}-\boldsymbol{x}_T){\times} \Curl \vec{c}\big).
  \]
  Using the vector calculus identity $\Curl(\boldsymbol{A}{\times} \boldsymbol{B}) = \boldsymbol{A}(\Div \boldsymbol{B}) - [\boldsymbol{A}{\cdot}\Grad]\boldsymbol{B} - (\Div \boldsymbol{A})\boldsymbol{B} + [\boldsymbol{B}{\cdot}\Grad]\boldsymbol{A}$ with $\big(\boldsymbol{A},\boldsymbol{B}\big) = \big(\boldsymbol{x}-\boldsymbol{x}_T,\Curl \vec{c}\big)$, we get
  \begin{multline*}
    \Curl \vec{p}
    = \cancel{(\boldsymbol{x}-\boldsymbol{x}_T)\Div (\Curl \vec{c})} - [(\boldsymbol{x}-\boldsymbol{x}_T){\cdot}\Grad](\Curl \vec{c}) 
    \\- 3 \Curl \vec{c} + [(\Curl \vec{c}){\cdot}\Grad](\boldsymbol{x}-\boldsymbol{x}_T),
  \end{multline*}
  where we have used the fact that the $\Div$ of the $\Curl$ is zero in the cancellation.
  Hence, observing that $[(\Curl \vec{c}){\cdot}\Grad]{(\boldsymbol{x}-\boldsymbol{x}_T)}
  = \big(\sum_{j=1}^3(\Curl\vec{c})_j\partial_j x_i\big)_{1\le i\le 3}
  = \big(\sum_{j=1}^3(\Curl\vec{c})_j\delta_{ij}\big)_{1\le i\le 3}
  = \Curl \vec{c}$ (with $\bullet_j$ denoting the $j$-th component of the vector $\bullet$ and $\delta_{ij}$ the Kronecker symbol),
  \begin{equation}\label{proof.ineq.decomp.1}
    \Curl \vec{p} = -2\Curl \vec{c} -[(\boldsymbol{x}-\boldsymbol{x}_T){\cdot}\Grad](\Curl \vec{c}).
  \end{equation}
  Since $[(\boldsymbol{x}-\boldsymbol{x}_T){\cdot}\Grad]{(\Curl \vec{c})}=\big(\sum_{j=1}^3(\boldsymbol{x}-\boldsymbol{x}_T)_j\partial_j(\Curl\vec{c})_i\big)_{1\le i\le 3}$, multiplying~\eqref{proof.ineq.decomp.1} by $-\Curl\vec{c}$ and integrating over $T$, we get
  \begin{equation}\label{proof.ineq.decomp.2}
  \begin{aligned}
    -\big(\Curl\vec{p},\Curl\vec{c}\big)_T 
    &= 2 \|\Curl \vec{c}\|_{T}^2 + \big(\boldsymbol{x}-\boldsymbol{x}_T,\Grad\left(\frac{|\Curl \vec{c}|^2}{2} \right)\big)_T
    \\
    &= \frac12\|\Curl \vec{c}\|_{T}^2 + \sum_{F\in\f_T} \frac{1}{2}\big(|\Curl \vec{c}|_{\mid F}^2,(\boldsymbol{x}-\boldsymbol{x}_T)_{\mid F}{\cdot}\normal_{TF}\big)_F 
    \\
    &\geq \frac12\|\Curl \vec{c}\|_{T}^2\geq 0,
  \end{aligned}
  \end{equation}
  where we have used an integration by parts formula to pass to the second line, and the fact that $T$ is star-shaped with respect to $\boldsymbol{x}_T$ to conclude. 
 From~\eqref{proof.ineq.decomp.2} and a Cauchy--Schwarz inequality, we infer
  \begin{equation}\label{proof.ineq.decomp.3}
    \|\Curl \vec{c}\|_{T}
    \leq 
    2\|\Curl \vec{p}\|_{T}.
  \end{equation}
  Plugging \eqref{proof.ineq.decomp.3} into \eqref{proof.ineq.decomp.0}, \eqref{lemma:ineq.decomp.eq} follows.
\end{proof}

We equip the space $\Xh$ with the seminorm $\|\cdot\|_{\boldsymbol{{\rm X}},h}$ defined by
\begin{equation}\label{def:norms.h:curl}
	\|\vh\|_{\boldsymbol{{\rm X}},h}^2 \defi \|\Curl_h\!\vTh \|_\Omega^2+ \sum_{T\in\t_h}\sum_{F\in\f_T} h_F^{-1}\|\vec{\pi}^{k+1}_{\boldsymbol{{\cal G}},F}\big(\vec{\gamma}_{\tau,F}(\vT) - \vF \big) \|_F^2.
\end{equation}
  This seminorm is the discrete counterpart of the $\boldsymbol L^2(\Omega;\mathbb{R}^3)$-norm of the $\Curl$ and is composed of two contributions: the first one is the $\boldsymbol L^2(\Omega;\mathbb{R}^3)$-norm of the broken $\Curl$ of the function obtained patching the polynomials attached to mesh elements; the second one accounts for the difference between the face unknowns and the (gradient part of the) tangential trace of the element unknowns. The kernel of this seminorm is made clear by Remark~\ref{rem:control}.
On the discrete counterpart of the space $\Hzcurl\cap\Hdivz$ defined by
\begin{equation} \label{def:XXhz}
  \XXhz\defi\left\{\wh\in\Xhz\st\big(\wTh,\Gkh\qh\big)_\Omega=0\quad\forall\qh\in\Yhz\right\},
\end{equation}
the following discrete Weber inequality holds true.
  Here and in what follows, the circle overset will be reserved to those hybrid spaces that incorporate a discrete divergence-free property (see, in particular, \eqref{def:ZZhz} and \eqref{def:general.ZZhz} below).
\begin{theorem}[Discrete Weber inequality]\label{thm:max.ineq}
	There exists a constant $c_{\rm W}>0$ independent of $h$ such that, for all $\vh\in \XXhz$, one has
	\begin{equation}\label{thm:max.ineq.eq}
		\| \vTh\|_\Omega \leq c_{\rm W}\|\vh\|_{\boldsymbol{{\rm X}},h},
	\end{equation}
  where we remind the reader that the term in the left-hand side is the $\boldsymbol L^2(\Omega;\mathbb{R}^3)$-norm of the broken polynomial vector $\vTh$ such that $(\vTh)_{|T}=\vT$ for all $T\in\t_h$.
\end{theorem}
\begin{proof}
  Let $\vh\in \XXhz$.
    Recall the following standard $\boldsymbol{L}^2(\Omega;\mathbb{R}^3)$-orthogonal Helmholtz decomposition (cf., e.g.,~\cite[Proposition 3.7.1]{Assous.ea:18}):
    $$\boldsymbol{L}^2(\Omega;\mathbb{R}^3)={\bf grad}\big(H^1_0(\Omega)\big)\overset{\perp}{\oplus}\Hdivz.$$
    By the characterization of divergence-free functions from~\cite[Theorem 3.4.1]{Assous.ea:18}, and since $\vTh\in \PkpdTh\subset \boldsymbol L^2(\Omega;\mathbb{R}^3)$ and $\partial\Omega$ is connected, we can write
        \begin{equation}\label{eq:vTh:orth.decomp}
          \vTh = \Grad \varphi + \Curl \boldsymbol{\psi},
        \end{equation}
for some $\varphi\in H^1_0(\Omega)$, and some $\boldsymbol{\psi}\in \boldsymbol H^1(\Omega;\mathbb{R}^3)$ such that $\int_\Omega\boldsymbol{\psi}=\boldsymbol{0}$ and $\Div\boldsymbol{\psi}=0$. Furthermore,
        \begin{equation} \label{eq:max.proof:bound.z}
          \|\boldsymbol{\psi}\|_{\boldsymbol H^1(\Omega;\mathbb{R}^3)}\lesssim\|\Curl\boldsymbol{\psi}\|_\Omega.
        \end{equation}
	Using the $\boldsymbol{L}^2(\Omega;\Real^3)$-orthogonal decomposition~\eqref{eq:vTh:orth.decomp} of $\vTh$, we have that
	\begin{equation}\label{eq:max.proof:norm.decomp}
		\|\vTh\|_\Omega^2 =
      (\vTh,\Grad \varphi)_\Omega
      + \big(\vTh,\Curl\boldsymbol{\psi}\big)_\Omega
    \ifed \mathcal{I}_1 + \mathcal{I}_2.
	\end{equation}
	
	For the first term in \eqref{eq:max.proof:norm.decomp}, setting $\underline{\upvarphi}_h \defi \IntYh \varphi\in\Yhz$, we infer that
	\begin{equation}\label{eq:max.proof:I1}
		\begin{aligned}
			\mathcal{I}_1
      = \big(\vTh,\boldsymbol{\pi}_{\boldsymbol{\cal P},h}^{k+1}(\Grad \varphi)\big)_\Omega = \big(\vTh,\Gkh \underline{\upvarphi}_h\big)_\Omega =  0,
		\end{aligned}
	\end{equation}
	where we have used that $\vTh\in \PkpdTh$ to insert $\boldsymbol{\pi}_{\boldsymbol{\cal P},h}^{k+1}$, followed by the commutation property~\eqref{lem:commut.GTkITk.eq} of $\Gkh$ and the fact that $\vh\in\XXhz$.
	
	Let us now estimate the second term in \eqref{eq:max.proof:norm.decomp}.
	We have
	\begin{equation*}
		\begin{aligned}
			\mathcal{I}_2 &=
      \sum_{T\in\t_h}\big(\vT,\Curl\boldsymbol{\psi}\big)_T
			= \sum_{T\in\t_h} \left(
				\big(\Curl \vT,\boldsymbol{\psi}\big)_T - \sum_{F\in\f_T}\big(\boldsymbol{\psi}_{\mid F}{\times}\normal_{TF},{\vT}_{\mid F}\big)_F
			\right)
			\\
			& = \sum_{T\in\t_h} \left(
				\big(\Curl \vT,\boldsymbol{\psi}\big)_T - \sum_{F\in\f_T}\big(\vec{\gamma}_{\tau,F}(\boldsymbol{\psi}{\times}\normal_{TF}),\vec{\gamma}_{\tau,F}(\vT)-\vF\big)_F
			\right),
		\end{aligned}
	\end{equation*}
        where we have used an integration by parts formula on each mesh element $T\in\t_h$ in the first line,
        and the fact that the jumps of $\boldsymbol{\psi}\in \boldsymbol H^1(\Omega;\mathbb{R}^3)$ vanish on interfaces along with $\vF=\boldsymbol{0}$ for all $F\in\f_h^{\rm b}$ to insert $\vF$ into the second term in the second line.
	Applying Cauchy--Schwarz inequalities to the right-hand side, we obtain
	\begin{equation}\label{eq:max.proof:I2a}
          \begin{split}
	    \mathcal{I}_2
	    \leq \Bigg(
	    \sum_{T\in\t_h}\Big(\|\Curl &\vT\|_T^2+\sum_{F\in\f_T} h_F^{-1}\|\vec{\gamma}_{\tau,F}(\vT)-\vF\|_F^2\Big)
	    \Bigg)^{\nicefrac{1}{2}}
	    \\
	    &\times\Bigg(
	    \sum_{T\in\t_h}\Big(\|\boldsymbol{\psi}\|_T^2+\sum_{F\in\f_T} h_F\|\boldsymbol{\psi}_{\mid F}{\times}\normal_{TF}\|_F^2\Big)
	    \Bigg)^{\nicefrac{1}{2}}.
          \end{split}
	\end{equation}
	We focus on the first factor on the right-hand side of \eqref{eq:max.proof:I2a}.
	For $T\in\t_h$ and $F\in\f_T$, decomposing $\vT\in\PkpdT$ along~\eqref{eq:decomp} as $\vT = \Grad g + {(\boldsymbol{x}-\boldsymbol{x}_T)}\times\Curl\vec{c}$ with $g\in\PkppT$ and $\vec{c}\in\PkpdT$, inserting into the norm $\pm\big[\vec{\gamma}_{\tau,F}(\Grad g)+\vec{\pi}_{\boldsymbol{{\cal G}},F}^{k+1}\big(\vec{\gamma}_{\tau,F}(\vT)\big)\big]$, and using the triangle inequality, we infer, since $\vF\in\boldsymbol{{\cal G}}^{k+1}(F)$ and $\vec{\gamma}_{\tau,F}(\Grad g) = \Grad_\tau (g_{\mid F})\in\boldsymbol{{\cal G}}^{k+1}(F)$,
	\begin{equation}\label{eq:crucial}
          \begin{split}
            \sum_{T\in\t_h}\!\sum_{F\in\f_T}\! h_F^{-1}\|\vec{\gamma}_{\tau,F}(\vT) - \vF\|_F^2
            \lesssim\sum_{T\in\t_h}\!\sum_{F\in\f_T}\! h_F^{-1}& \|\vec{\pi}_{\boldsymbol{{\cal G}},F}^{k+1}\big(\vec{\gamma}_{\tau,F}(\vT)-\vF\big) \|_F^2
            \\
	    +\sum_{T\in\t_h}\!\sum_{F\in\f_T}&\! h_F^{-1}\|\vec{\gamma}_{\tau,F}\big(\vT- \Grad g\big) \|_F^2
            \\
            +\sum_{T\in\t_h}\!\sum_{F\in\f_T}\! h_F^{-1}\|&\vec{\pi}_{\boldsymbol{{\cal G}},F}^{k+1}\big(\vec{\gamma}_{\tau,F}(\vT- \Grad g)\big) \|_F^2.
          \end{split}
        \end{equation}
        Using the $\boldsymbol L^2(F;\mathbb{R}^2)$-boundedness of $\vec{\pi}_{\boldsymbol{{\cal G}},F}^{k+1}$, a discrete trace inequality (cf., e.g.,~\cite[Lemma 1.32]{Di-Pietro.Droniou:20}), and Lemma \ref{lemma:ineq.decomp} with $\vec{p}=\vT$, we infer
	\begin{equation}\label{eq:max.proof:I2b}
          \begin{split}
	    \sum_{T\in\t_h}\sum_{F\in\f_T}&h_F^{-1}\|\vec{\gamma}_{\tau,F}(\vT)-\vF\|_F^2
	    \\
	    &\lesssim \sum_{T\in\t_h}\Big(\|\Curl \vT\|_T^2+\sum_{F\in\f_T} h_F^{-1}\|\vec{\pi}_{\boldsymbol{{\cal G}},F}^{k+1}\big(\vec{\gamma}_{\tau,F}(\vT)-\vF\big)\|_F^2\Big).
          \end{split}
	\end{equation}
	Now, for the second factor on the right-hand side of \eqref{eq:max.proof:I2a}, using that $|\vec{\psi}_{\mid F}{\times}\normal_{TF}|\leq|\vec{\psi}_{\mid F}|$, a continuous trace inequality (cf., e.g.,~\cite[Lemma 1.31]{Di-Pietro.Droniou:20}), the fact that $h_T\leq {\rm diam}(\Omega)$ for all $T\in\t_h$, and concluding with~\eqref{eq:max.proof:bound.z}, one has
	\begin{equation}\label{eq:max.proof:I2c}
		\left(
		\sum_{T\in\t_h}\Big(\|\boldsymbol{\psi}\|_T^2
                +\!\!\sum_{F\in\f_T} h_F\|\boldsymbol{\psi}_{\mid F}{\times}\normal_{TF}\|_F^2\Big)
		\right)^{\nicefrac{1}{2}}
		\!\!\lesssim\|\boldsymbol{\psi}\|_{\boldsymbol H^1(\Omega;\mathbb{R}^3)}\lesssim\|\Curl\boldsymbol{\psi}\|_\Omega.
	\end{equation}
	Plugging \eqref{eq:max.proof:I2b} and \eqref{eq:max.proof:I2c} into \eqref{eq:max.proof:I2a}, and recalling the definition \eqref{def:norms.h:curl} of the $\|\cdot\|_{\boldsymbol{{\rm X}},h}$-seminorm yields
          \begin{equation*}\label{eq:ineq.decomp:I2}
          \mathcal{I}_2
          \lesssim\|\vh\|_{\boldsymbol{{\rm X}},h}\|\Curl\vec{\psi}\|_\Omega
          \le\|\vh\|_{\boldsymbol{{\rm X}},h}\|\vTh\|_\Omega,
          \end{equation*}
          where we have used the $\boldsymbol L^2(\Omega;\mathbb{R}^3)$-orthogonality of the decomposition~\eqref{eq:vTh:orth.decomp} in the last bound.
	We conclude by combining \eqref{eq:max.proof:norm.decomp}, \eqref{eq:max.proof:I1}, and this last estimate.
\end{proof}

\begin{remark}[Control of element unknowns]\label{rem:control}
  A direct proof of the fact that, for all $\vh\in\XXhz$, $\|\vh\|_{\boldsymbol{{\rm X}},h}=0$ implies $\|\vTh\|_{\Omega}=0$ can be obtained as follows.
  The volumetric term in~\eqref{def:norms.h:curl} first yields that, for all $T\in\t_h$, $\Curl(\vec{{\rm v}}_{h\mid T})=\vec{0}$ in $T$, meaning that $\vec{{\rm v}}_{h\mid T}=\Grad g$ for some $g\in\PkppT$ by Lemma~\ref{lemma:ineq.decomp}.
  The boundary term in~\eqref{def:norms.h:curl} then yields the continuity of the tangential component of $\vTh$ at interfaces, as well as $\vec{n}{\times}(\vTh{\times}\vec{n})=\vec{0}$ on $\partial\Omega$. Hence, $\vTh\in\Hzcurl$ and $\Curl\vTh=\vec{0}$ in $\Omega$. Since $\vh\in\XXhz$, we also have $\vTh\in\Hdivz$ by the commutation property~\eqref{lem:commut.GTkITk.eq}. Finally, by the continuous first Weber inequality~\eqref{eq:weber}, $\|\vTh\|_{\Omega}=0$.
  Notice that this result is weaker than the quantitative estimate of Theorem~\ref{thm:max.ineq}, as it does not give any information on how the constant $c_{\rm W}$ depends on the mesh at hand.
  Notice also that imposing $\vTh\in\Hdivz$ as we do is actually not necessary. In view of the above analysis, it is sufficient and necessary to impose that $\vTh$ be orthogonal to the gradient of any function in $\Poly^{k+2}(\t_h)\cap C^0_0(\Omega)$. This is the approach pursued in~\ifMMMAS\else\cite{Chen.ea:18} and~\cite{Chen.Monk.ea:20} \fi on tetrahedral meshes.
\end{remark}

\begin{corollary}[Norm $\|\cdot\|_{\boldsymbol{{\rm X}},h}$] \label{co:norm.X}
	The map $\|\cdot\|_{\boldsymbol{{\rm X}},h}$ defines a norm on $\XXhz$ defined by~\eqref{def:XXhz}.
\end{corollary}
\begin{proof}
  This is a direct consequence of Theorem~\ref{thm:max.ineq} and of the definition~\eqref{def:norms.h:curl}. For all $\vh\in\XXhz$, if $\|\vh\|_{\boldsymbol{{\rm X}},h}=0$, then $\vTh=\vec{0}$, i.e.~$\vT=\vec{0}$ for all $T\in\t_h$. Then, for all $F\in\f_h$, $\|\vec{\pi}_{\boldsymbol{{\cal G}},F}^{k+1}(\vF)\|_{F}=\|\vF\|_{F}=0$, i.e.~$\vF=\vec{0}$, whence $\vh=\underline{\vec{0}}_h$.
\end{proof}

\begin{corollary}[Generalized discrete Weber inequality]\label{cor:max.ineq}
  Let $\dh:\Yh\times\Yh\to\Real$ be a symmetric positive semi-definite bilinear form such that, for all $\varphi\in H^1_0(\Omega)$, letting $\underline{\upvarphi}_h\defi\IntYh\varphi\in\Yhz$,
  \begin{equation} \label{eq:approx}
    \dh\big(\underline{\upvarphi}_h,\underline{\upvarphi}_h\big)^{\nicefrac12}\lesssim\|\Grad\varphi\|_{\Omega}.
  \end{equation}
  Then, there is $c_{\rm W}>0$ independent of $h$ such that, for all $(\vh,\rh)\in\Xhz{\times}\Yhz$ satisfying
	\begin{equation}\label{rmk:max.ineq.cond}
		-\big(\vTh,\Gkh\qh\big)_\Omega + \dh\big(\rh,\qh\big) =  0\qquad\forall\qh\in\Yhz,
	\end{equation}
	one has
	\begin{equation}\label{rmk:max.ineq.eq}
		\| \vTh\|_\Omega \leq c_{\rm W}\Big( \|\vh\|^2_{\boldsymbol{{\rm X}},h} + \dh(\rh,\rh)\Big)^{\nicefrac{1}{2}}.
	\end{equation}
\end{corollary}
\begin{proof}
  We follow the steps of the proof of Theorem~\ref{thm:max.ineq}.
  If $(\vh,\rh)\in\Xhz{\times}\Yhz$ satisfies \eqref{rmk:max.ineq.cond}, then \eqref{eq:max.proof:I1} becomes
  \[
  \mathcal{I}_1 = (\vTh,\Grad \varphi)_\Omega = \big(\vTh,\Gkh\underline{\upvarphi}_h\big)_\Omega =  \dh\big(\rh,\underline{\upvarphi}_h\big).
  \]
  By the Cauchy--Schwarz inequality and~\eqref{eq:approx}, we infer
  \begin{equation} \label{eq:last.ineq}
    \mathcal{I}_1
    \lesssim\dh(\rh,\rh)^{\nicefrac{1}{2}}\|\Grad\varphi\|_\Omega
    \leq\dh(\rh,\rh)^{\nicefrac{1}{2}}\|\vTh\|_\Omega,
  \end{equation}
  where we have used the $\boldsymbol L^2(\Omega;\mathbb{R}^3)$-orthogonality of the decomposition~\eqref{eq:vTh:orth.decomp} in the last bound.
  The rest of the proof is unchanged provided we substitute~\eqref{eq:last.ineq} to \eqref{eq:max.proof:I1}.
  \end{proof}
\noindent
This last corollary will be instrumental in the analysis of the HHO method of Section~\ref{sec:general}. Before closing this section, two additional remarks are in order.
\begin{remark}[Topological assumptions on the domain] \label{rem:top}
  The first Weber inequality~\eqref{eq:weber} is actually valid under the sole topological assumption that the boundary of $\Omega$ is connected, so that its second Betti number is zero. The same holds true for the discrete Weber inequalities of Theorem~\ref{thm:max.ineq} and Corollary~\ref{cor:max.ineq} (and, incidentally, this is also the case for discrete Weber inequalities on spaces with conforming unknowns, see~\cite[Theorem~19]{Di-Pietro.Droniou:20*1}). In other words, one does not need to assume, as we do, that $\Omega$ is simply-connected to prove these results. This last assumption is however necessary in the applicative Section~\ref{se:magneto} to have equivalence \begin{inparaenum}[(i)] \item between Problems~\eqref{eq:standard.strong} and~\eqref{eq:standard.weak} in field formulation, and \item between Problems~\eqref{eq:strong} and~\eqref{eq:strong.equiv} (in the class of potentials satisfying the Coulomb gauge) in vector potential formulation when $\Div\vec{f} = 0$\end{inparaenum}.
\end{remark}
\begin{remark}[Star-shapedness assumption]\label{rem:role.star-shaped}
  We have assumed in Section~\ref{se:disset} that the mesh elements are star-shaped. This assumption has been instrumental to prove Lemma~\ref{lemma:ineq.decomp} (and, in turn, Theorem~\ref{thm:max.ineq} and Corollary~\ref{cor:max.ineq}), where it is used to infer a sign for the rightmost term in the second line of~\eqref{proof.ineq.decomp.2}.
  According to~\cite[Lemma 46]{Di-Pietro.Droniou:21}, this assumption is in fact not necessary for Lemma~\ref{lemma:ineq.decomp} to hold true, and all the results of this article actually (seamlessly) extend to the general case of meshes featuring non-necessarily star-shaped mesh elements (in that case, $\vec{x}_T$ denotes any given interior point of $T$ such that $T$ contains a ball centered at $\vec{x}_T$ of radius comparable to $h_T$).
  Nonetheless, we have preferred to keep this assumption in our analysis because it enables, as opposed to~\cite[Lemma 46]{Di-Pietro.Droniou:21}, whose proof hinges on a transport argument, to obtain an explicit multiplicative constant in front of $h_T$ in~\eqref{lemma:ineq.decomp.eq}.
\end{remark}


\section{Application to magnetostatics} \label{se:magneto}

In this section, we design and analyze HHO methods for the discretization of the magnetostatics equations.
Their analysis leverages the discrete Weber inequality of Theorem~\ref{thm:max.ineq} and its generalization pointed out in Corollary~\ref{cor:max.ineq}.
For the sake of simplicity, we focus on homogeneous boundary conditions: the extension of HHO methods to non-homogeneous boundary conditions is standard (see, e.g., \cite[Section 2.4]{Di-Pietro.Droniou:20}) and will be considered in the numerical example of Section~\ref{sec:standard:num}.
We recall that we work on regular (polyhedral) mesh sequences $({\cal M}_h)_{h>0}$ in the sense of~\cite[Definition 1.9]{Di-Pietro.Droniou:20}, which are characterized by the fact that the sequence of mesh regularity parameters is bounded from below by a strictly positive real number.

\subsection{Field formulation} \label{sse:field}

\subsubsection{The model}

The (first-order) field formulation of the magnetostatics problem consists in finding the magnetic field $\boldsymbol{u}:\Omega\rightarrow\Real^3$ such that
\begin{subequations}
  \label{eq:standard.strong}
  \begin{alignat}{2}
    \Curl \boldsymbol{u} &= \boldsymbol{f} &\qquad&\text{in $\Omega$},\label{eq:standard.strong:1}
    \\
    \Div \boldsymbol{u} &= 0 &\qquad&\text{in $\Omega$},\label{eq:standard.strong:2}
	\\
    \normal{\times}(\boldsymbol{u}{\times}\normal)  &=  \boldsymbol{0} &\qquad&\text{on $\partial \Omega$},\label{eq:standard.strong:bc.u} 
  \end{alignat}
\end{subequations}
where the current density $\boldsymbol{f}:\Omega\to\Real^3$ is such that $\Div \boldsymbol{f} = 0$ in $\Omega$ and $\boldsymbol{f}{\cdot}\normal = 0$ on $\partial\Omega$. 
We consider the following equivalent (cf.~Remark~\ref{rem:top}) weak formulation of Problem \eqref{eq:standard.strong}, originally introduced in~\cite[Eq.~(52)]{Kikuchi:89} (see also~\ifMMMAS\else\cite{Kanayama.ea:90}\fi):
Find $(\boldsymbol{u},p)\in\Hzcurl\times H^1_0(\Omega)$ such that
\begin{subequations}\label{eq:standard.weak}
	\begin{alignat}{2}\label{eq:standard.weak.1}
		a(\boldsymbol{u},\boldsymbol{v}) + b(\boldsymbol{v},p) & = (\boldsymbol{f},\Curl \boldsymbol{v})_\Omega &\qquad&\forall \boldsymbol{v}\in\Hzcurl,
		\\\label{eq:standard.weak.2}
		b(\boldsymbol{u},q) & = 0 &\qquad&\forall q\in H^1_0(\Omega),
	\end{alignat}
\end{subequations}
where the bilinear forms $a:\Hcurl\times\Hcurl\rightarrow\Real$ and $b:\Hcurl\times H^1(\Omega) \rightarrow\Real$ are given by
\begin{equation}\label{def:standard.bf.c0}
	a(\boldsymbol{w},\boldsymbol{v}) \defi (\Curl \boldsymbol{w},\Curl \boldsymbol{v})_\Omega,
	\qquad
	b(\boldsymbol{w},q) \defi (\boldsymbol{w},\Grad q)_\Omega.
\end{equation}
The function $p:\Omega\rightarrow\Real$ is the Lagrange multiplier of the divergence-free constraint on the magnetic induction.
Testing~\eqref{eq:standard.weak.1} with $\vec{v}=\Grad p\in\Hzcurl$, it is readily inferred that $p=0$ in $\Omega$.
By the decomposition~\eqref{eq:helmholtz}, Problem~\eqref{eq:standard.weak} can then be equivalently rewritten: Find $\vec{u}\in\Hzcurl\cap\Hdivz$ such that
$$a(\vec{u},\vec{\eta})=(\boldsymbol{f},\Curl \boldsymbol{\eta})_\Omega\qquad\forall\boldsymbol{\eta}\in\Hzcurl\cap\Hdivz,$$
whose well-posedness is a direct consequence of the first Weber inequality~\eqref{eq:weber} and of the Lax--Milgram lemma.

\subsubsection{The HHO method}

We analyze in this section the HHO method for Problem~\eqref{eq:standard.weak} we have briefly introduced in~\ifMMMAS\else\cite{Chave.ea:20}\fi.
This HHO method is based on the hybrid spaces introduced in Section~\ref{sse:hybrid} (the space $\Xh$ for the magnetic field, and $\Yh$ for the Lagrange multiplier).
We define the discrete bilinear forms $\ah:\Xh\times\Xh\rightarrow\Real$, $\bh:\Xh\times\Yh\rightarrow\Real$, and $\ch:\Yh\times\Yh\rightarrow\Real$ such that
\begin{subequations}\label{def:standard.bf.h}
	\begin{alignat}{2}\label{def:standard.bf.ah}
		\ah(\wh,\vh) & \defi \big(\Curl_h\!\wTh,\Curl_h\!\vTh\big)_\Omega + \sh(\wh,\vh),
		\\\label{def:standard.bf.bh}
		\bh(\wh,\qh) & \defi \big(\wTh,\Gkh\qh\big)_\Omega,
		\\\label{def:standard.bf.ch}
		\ch(\rh,\qh) & \defi (\rTh,\qTh)_\Omega + \sum_{T\in\t_h}\sum_{F\in\f_T} h_F (\rF,\qF)_F,
	\end{alignat}
\end{subequations}
where $\Gkh:\Yh\rightarrow\PkpdTh$ is the gradient reconstruction operator introduced in Section \ref{sec:grad.recons}, and $\sh:\Xh\times\Xh\rightarrow\Real$ is the stabilization bilinear form such that
\begin{equation}\label{def:standard.bf.sh}
  \sh(\wh,\vh)\defi \sum_{T\in\t_h}\sum_{F\in\f_T} h_F^{-1}\big(\vec{\pi}_{\boldsymbol{{\cal G}},F}^{k+1}\big(\vec{\gamma}_{\tau,F}(\wT)-\wF\big),\vec{\pi}_{\boldsymbol{{\cal G}},F}^{k+1}\big(\vec{\gamma}_{\tau,F}(\vT)-\vF\big)\big)_F.
\end{equation}
The HHO method for Problem~\eqref{eq:standard.weak} then reads:
Find $(\uh,\ph)\in\Xhz\times\Yhz$ such that
\begin{subequations}\label{eq:standard.discrete}
	\begin{alignat}{3}\label{eq:standard.discrete.1}
		\ah(\uh,\vh) + \bh(\vh,\ph) & = (\boldsymbol{f},\Curl_h\!\vTh)_\Omega&\qquad&\forall \vh\in\Xhz,
		\\\label{eq:standard.discrete.2}
		-\bh(\uh,\qh) +\ch(\ph,\qh)& = 0&\qquad&\forall\qh\in\Yhz.
	\end{alignat}
\end{subequations}
Notice that, contrary to $p$, the discrete Lagrange multiplier $\ph$ is in general nonzero, as a consequence of the fact that the global discrete gradient $\Gkh\ph$ is not irrotational. Some remarks are in order.
\begin{remark}[The tetrahedral case]\label{rk:tet.case.1}
  On matching tetrahedral meshes, according to Lemma~\ref{le:GTk.norm}, $\|\Gkh\cdot\|_{\Omega}$ defines a norm on $\Yhz$. Hence, in this case, one can consider a modified version of Problem~\eqref{eq:standard.discrete} in which $\ch$ is removed and for which stability (hence well-posedness) is preserved. This will be justified rigorously by Lemma~\ref{lem:infsupbh} below (which essentially states an inf-sup condition for $\bh$). Removing $\ch$, $\uh\in\XXhz$ holds true and, as a by-product of the commutation property~\eqref{lem:commut.GTkITk.eq}, $\uTh\in\Hdivz$.
\end{remark}
\begin{remark}[Improved stability] \label{rem:stab}
  On general mesh families, and as opposed to the tetrahedral case, one has to add in~\eqref{eq:standard.discrete.2} a (consistent) positive semi-definite stabilization of the Lagrange multiplier to ensure the stability of the method. An example of such a stabilization is given by $\dh$ defined by~\eqref{def:general.bf.ch}.
  Here, taking advantage of the fact that the continuous Lagrange multiplier $p$ is identically zero (as a consequence of~\eqref{eq:standard.weak.1}), we choose to add the (consistent) positive-definite contribution $\ch(\ph,\qh)$ (remark that $\ch$ defines a norm on $\Yh$). This strategy enables to improve the stability of the method without deteriorating its convergence properties. At the opposite, in the model of Section~\ref{sec:general} below, the Lagrange multiplier may be nonzero and one cannot add the same contribution at the discrete level because it is not consistent anymore. One uses instead the semi-definite (consistent) bilinear form $\dh$ of~\eqref{def:general.bf.ch} and can then only prove a weaker stability result (compare Lemma~\ref{le:wp} and Lemma~\ref{thm:wellposedness}). To assess the effect of using $\ch$ as a stabilization, we compare in Section~\ref{sec:standard:num} (see Figure~\ref{fig:standard.error.tetra.nostab}) the numerical results obtained on a matching tetrahedral mesh family for Problem~\eqref{eq:standard.discrete} with $\ch$ and without $\ch$ (which is possible according to Remark~\ref{rk:tet.case.1}). Always better results (in absolute value) are obtained when using $\ch$.
\end{remark}
\begin{remark}[Curl reconstruction] \label{rem:curl}
  As opposed to what is done in HHO
  methods for second-order problems (see Section~\ref{sec:general} below), we here take advantage of the fact that the problem is first-order to avoid (locally) reconstructing a discrete $\Curl$ operator. Doing so, \begin{inparaenum}[(i)] \item it is possible to consider a smaller local space of face unknowns (that does not need to contain $\boldsymbol\Poly^k(F)$) for $k\geq 1$ (cf.~\cite[Table~1]{Chave.ea:20}), and \item there is no need to solve a local problem on each mesh element (which may become, for a sequential implementation, rather costly in 3D for large polynomial degrees). \end{inparaenum}
\end{remark}

Letting $\Zh\defi\Xh\times\Yh$ and $\Zhz\defi\Xhz\times\Yhz$, Problem~\eqref{eq:standard.discrete} can be equivalently rewritten: Find $(\uh,\ph)\in\Zhz$ such that
\begin{equation}\label{eq:standard.discrete.equiv}
  \Ah\big((\uh,\ph),(\vh,\qh)\big)=(\boldsymbol{f},\Curl_h\!\vTh)_\Omega\qquad\forall\,(\vh,\qh)\in\Zhz,
\end{equation}
where the bilinear form $\Ah:\Zh\times\Zh\to\Real$ is defined by
\begin{equation}\label{def:standard.Ah}
	\Ah\big((\wh,\rh),(\vh,\qh)\big) \defi \ah(\wh,\vh) + \bh(\vh,\rh) - \bh(\wh,\qh) + \ch(\rh,\qh).
\end{equation}
For future use, we also let
\begin{equation}\label{def:ZZhz}
  \begin{split}
    \ZZhz\defi&\left\{(\wh,\rh)\in\Zhz\,:\,-\bh(\wh,\qh) +\ch(\rh,\qh) = 0 \quad\forall\qh\in\Yhz\right\}\\
    =&\left\{(\wh,\rh)\in\Zhz\,:\,\Ah\big((\wh,\rh),(\underline{\vec{0}}_h,\qh)\big) = 0 \quad\forall\qh\in\Yhz\right\}.
  \end{split}
\end{equation}

\subsubsection{Stability analysis} \label{ssse:stab}

We recall that $\Xh$ is equipped with the seminorm $\|\cdot\|_{\boldsymbol{{\rm X}},h}$ defined by~\eqref{def:norms.h:curl}, which is such that $\|\cdot\|_{\boldsymbol{{\rm X}},h}=\ah(\cdot,\cdot)^{\nicefrac12}$.
We equip $\Yh$ with the norm
\begin{equation}\label{eq:norm.Y.h}
  \|\cdot\|_{{\rm Y},h}\defi\ch(\cdot,\cdot)^{\nicefrac12},
\end{equation}
and $\Zh$ with the seminorm
\begin{equation}\label{def:standard.norm}
  \|(\wh,\rh)\|_{\mathbb{Z},h} \defi\left(
  \|\wh\|_{\boldsymbol{{\rm X}},h}^2+\|\rh\|_{{\rm Y},h}^2
  \right)^{\nicefrac12}.
\end{equation}
\begin{lemma}[Norm $\|\cdot\|_{\mathbb{Z},h}$] \label{le:norm.Z}
	The map $\|\cdot\|_{\mathbb{Z},h}$ defines a norm on $\ZZhz$ defined by~\eqref{def:ZZhz}.
\end{lemma}

\begin{proof}
  The seminorm property being straightforward, it suffices to prove that, for all $(\wh,\rh)\in\ZZhz$, $\|(\wh,\rh)\|_{\mathbb{Z},h}=0$ implies $(\wh,\rh)=(\underline{\boldsymbol{0}}_h,\underline{0}_h)$.
	Let then $(\wh,\rh)\in\ZZhz$ satisfying $\|(\wh,\rh)\|_{\mathbb{Z},h} = 0 $. We infer that $\|\wh\|_{\boldsymbol{{\rm X}},h}=0$ and $\|\rh\|_{{\rm Y},h}=0$.
        Since $\|\cdot\|_{{\rm Y},h}$ is a norm on $\Yh$, the second relation directly implies that $\rh=\underline{0}_h$.
        Now, owing to the definitions~\eqref{def:ZZhz},~\eqref{def:standard.bf.bh}, and~\eqref{def:XXhz}, since $(\wh,\rh)\in\ZZhz$ and $\rh=\underline{0}_h$, we have $\wh\in\XXhz$. By Corollary~\ref{co:norm.X}, this implies $\wh=\underline{\vec{0}}_h$.
\end{proof}
\begin{lemma}[Well-posedness] \label{le:wp}
For all $\zh\in\Zh$,
  \begin{equation} \label{eq:coer}
    \Ah\big(\zh,\zh\big)=\|\zh\|_{\mathbb{Z},h}^2.
  \end{equation}
  Hence, Problem~\eqref{eq:standard.discrete} is well-posed, and the following a priori bound holds true:
  \begin{equation}\label{eq:standard.discrete:a-priori}
    \|(\uh,\ph)\|_{\mathbb{Z},h}\leq\|\vec{f}\|_{\Omega}.
  \end{equation}
\end{lemma}
\begin{proof}
  The identity~\eqref{eq:coer} is a direct consequence of~\eqref{def:standard.Ah} and~\eqref{def:standard.norm} along with the definitions of $\|\cdot\|_{\boldsymbol{\rm X},h}$ and $\|\cdot\|_{{\rm Y},h}$.
  To prove well-posedness, since the linear system associated to Problem~\eqref{eq:standard.discrete} is square, it is sufficient to prove injectivity. Assume that $\Ah\big((\uh,\ph),(\vh,\qh)\big)=0$ for all $(\vh,\qh)\in\Zhz$. Choosing $(\vh,\qh)=(\underline{\vec{0}}_h,\qh)$ and using~\eqref{def:ZZhz}, we first infer that $(\uh,\ph)\in\ZZhz$. Choosing $(\vh,\qh)=(\uh,\ph)$ and using~\eqref{eq:coer}, we then get
  $$\|(\uh,\ph)\|_{\mathbb{Z},h}^2=\Ah\big((\uh,\ph),(\uh,\ph)\big)=0,$$
  which, by Lemma~\ref{le:norm.Z}, eventually yields $(\uh,\ph)=(\underline{\vec{0}}_h,\underline{0}_h)$. The a priori bound~\eqref{eq:standard.discrete:a-priori} directly follows from~\eqref{eq:coer} with $\zh=(\uh,\ph)$,~\eqref{eq:standard.discrete.equiv}, the Cauchy--Schwarz inequality, and $\|\uh\|_{\boldsymbol{{\rm X}},h}\leq\|(\uh,\ph)\|_{\mathbb{Z},h}$.
\end{proof}

\subsubsection{Error analysis}

We recall that $(\vec{u},p)\in\Hzcurl\times H^1_0(\Omega)$ denotes the unique solution to Problem~\eqref{eq:standard.weak}.
We assume from now on that $\vec{u}$ possesses the additional regularity $\vec{u}\in \boldsymbol H^1(\Omega;\mathbb{R}^3)$, and we let $\hatuh \defi \IntUh\boldsymbol{u}\in\Xhz$ and $\hatph\defi \IntPh p\in\Yhz$.
In the spirit of~\ifMMMAS\else\cite{Di-Pietro.Droniou:18} \fi (see also~\cite[Appendix~A]{Di-Pietro.Droniou:20}), we estimate the errors
\begin{equation}\label{def:standard.errors}
	\Xhz\ni\erruh\defi \uh-\hatuh,
	\qquad
	\Yhz\ni\errph\defi \ph - \hatph,
\end{equation}
where $(\uh,\ph)\in\Xhz\times\Yhz$ is the unique solution to Problem~\eqref{eq:standard.discrete}.
Notice that, since $p=0$ in $\Omega$, we actually have $\hatph=\underline{0}_h$ and $\errph = \ph$.
Recalling \eqref{eq:standard.discrete.equiv} and \eqref{def:standard.Ah}, the errors $(\erruh,\errph)\in\Zhz$ solve
\begin{equation}\label{eq:standard.discrete:Ah:error}
	\Ah\big((\erruh,\errph),(\vh,\qh)\big) = \lh(\vh) + \mh\big(\qh\big) \qquad \forall (\vh,\qh)\in\Zhz,
\end{equation}
where we have defined the consistency error linear forms
\begin{subequations}\label{def:standard.consistency.errors}
	\begin{alignat}{1}\label{def:standard.consistency.errors.lh}
		\lh(\vh) &\defi (\boldsymbol{f},\Curl_h\!\vTh)_\Omega - \ah(\hatuh,\vh),
		\\\label{def:standard.consistency.errors.mh}
		\mh\big(\qh\big) & \defi \bh(\hatuh,\qh).
	\end{alignat}
\end{subequations}

\begin{theorem}[Energy-error estimate]\label{thm:standard.error}
  Assume that
  $$\vec{u}\in\Hzcurl\cap \boldsymbol H^1(\Omega;\mathbb{R}^3)\cap \boldsymbol H^{k+2}(\t_h;\mathbb{R}^3).$$
  Then, the following holds true, with $(\erruh,\errph)\in\Zhz$ defined by \eqref{def:standard.errors}:
  \begin{equation}\label{thm:standard.error.eq}
    \|(\erruh,\errph)\|_{\mathbb{Z},h} \lesssim \left(\sum_{T\in\t_h} h_T^{2(k+1)}|\boldsymbol{u}|_{\boldsymbol H^{k+2}(T;\mathbb{R}^3)}^2\right)^{\nicefrac{1}{2}}.
  \end{equation}
\end{theorem}

\begin{proof}\label{proof:standard.error}
  Since $(\erruh,\errph)\in\Zhz$, by~\eqref{eq:coer} with $\zh=(\erruh,\errph)$ and~\eqref{eq:standard.discrete:Ah:error}, we infer
	\begin{equation}\label{proof:standard.error.2}
		\|(\erruh,\errph)\|_{\mathbb{Z},h} \leq \max_{(\vh,\qh)\in\Zhz,\|(\vh,\qh)\|_{\mathbb{Z},h}=1}\big(\lh(\vh) + \mh\big(\qh\big)\big).
	\end{equation}

        Let us first focus on $\lh(\vh)$ for $\vh\in\Xhz$. Combining its definition~\eqref{def:standard.consistency.errors.lh} with the fact that $\vec{f}=\Curl\vec{u}$ in $\Omega$, and the definition~\eqref{def:Icurl.h} of $\IntUh\boldsymbol{u}$, we infer
	\begin{equation*}
	  \begin{aligned}
	    |\lh(\vh)|
            &= \left|
            \big(\Curl\vec{u}-\Curl_h(\vec{\pi}_{\boldsymbol{\cal P},h}^{k+1}\vec{u}),\Curl_h\!\vTh\big)_{\Omega}-\sh(\hatuh,\vh)
            \right|
		\\
		&\leq \Big(
                \|\Curl\vec{u}-\Curl_h(\vec{\pi}_{\boldsymbol{\cal P},h}^{k+1}\vec{u})\|^2_{\Omega}
                  + \sh(\hatuh,\hatuh)
                  \Big)^{\nicefrac12}\|\vh\|_{\boldsymbol{{\rm X}},h},
	  \end{aligned}
	\end{equation*}	 
	where we have used the triangle/Cauchy--Schwarz inequalities and the definition~\eqref{def:norms.h:curl} of $\|\vh\|_{\boldsymbol{{\rm X}},h}$ to pass to the second line.
        The quantity $\|\Curl\vec{u}-\Curl_h(\vec{\pi}_{\boldsymbol{\cal P},h}^{k+1}\vec{u})\|^2_{\Omega}$ is estimated using the approximation properties of $\vec{\pi}_{\boldsymbol{\cal P},h}^{k+1}$ (see, e.g.,~\cite[Theorem 1.45]{Di-Pietro.Droniou:20}). For the quantity $\sh(\hatuh,\hatuh)$, recalling the definition~\eqref{def:standard.bf.sh} of $\sh$ and using the $\boldsymbol L^2(F;\mathbb{R}^2)$-boundedness of $\vec{\pi}_{\boldsymbol{{\cal G}},F}^{k+1}$, we have
	\begin{equation}\label{eq:hatuh}
          \begin{split}
		\sh(\hatuh,&\hatuh) = \sum_{T\in\t_h}\sum_{F\in\f_T} h_F^{-1}\|\vec{\pi}_{\boldsymbol{{\cal G}},F}^{k+1}\big(\vec{\gamma}_{\tau,F}(\vec{\pi}_{\boldsymbol{\cal P},T}^{k+1}(\boldsymbol{u}_{\mid T}) - {\boldsymbol{u}})\big) \|_F^2
		\\
		&\leq \sum_{T\in\t_h}\sum_{F\in\f_T} h_F^{-1}\|\vec{\pi}_{\boldsymbol{\cal P},T}^{k+1}(\boldsymbol{u}_{\mid T}) - \boldsymbol{u} \|_F^2
		\lesssim \sum_{T\in\t_h}h_T^{2(k+1)}|\boldsymbol{u}|_{\boldsymbol H^{k+2}(T;\mathbb{R}^3)}^2,
          \end{split}
	\end{equation}
	where, for all $T\in\t_h$, we have used the approximation properties of $\vec{\pi}_{\boldsymbol{\cal P},T}^{k+1}$ on the faces of $T$. Gathering the different estimates, we get
        \begin{equation} \label{eq:lh}
          |\lh(\vh)|
          \lesssim\left(\sum_{T\in\t_h}h_T^{2(k+1)}|\boldsymbol{u}|_{\boldsymbol H^{k+2}(T;\mathbb{R}^3)}^2\right)^{\nicefrac12}\|\vh\|_{\boldsymbol{{\rm X}},h}.
        \end{equation}
	
	Let us now focus on $\mh\big(\qh\big)$ for $\qh\in\Yhz$. Starting from~\eqref{def:standard.consistency.errors.mh}, performing an element-by-element integration by parts in~\eqref{def:GTk}, and using that $\Grad \qT\in\boldsymbol{{\cal G}}^{k-1}(T)\subset\PkpdT$, we infer
	\begin{align*}\notag
		\mh\big(\qh\big) &= \sum_{T\in\t_h}\Bigg(
			(\Grad \qT,\boldsymbol{u})_T + \sum_{F\in\f_T}\big(\vec{\pi}_{\boldsymbol{\cal P},T}^{k+1}(\boldsymbol{u}_{\mid T})_{\mid F}{\cdot}\normal_{TF},\qF-{\rm q}_{T\mid F}\big)_F
		\Bigg)
		\\
		&=\sum_{T\in\t_h}\sum_{F\in\f_T}\big((\vec{\pi}_{\boldsymbol{\cal P},T}^{k+1}(\boldsymbol{u}_{\mid T})-\boldsymbol{u})_{\mid F}{\cdot}\normal_{TF},\qF-{\rm q}_{T\mid F}\big)_F, 
	\end{align*}
	where the last identity follows from another element-by-element integration by parts, and from the fact that $\Div\vec{u}=0$ in $\Omega$, and that $\boldsymbol{u}\in \boldsymbol H^1(\Omega;\mathbb{R}^3)$ along with $\qF=0$ for all $F\in\f_h^{\rm b}$.
	By the triangle and Cauchy--Schwarz inequalities, one then gets
	\begin{equation}\label{eq:bh.hatuh}
          \begin{split}
		|\mh\big(\qh\big)|\leq \Bigg(
			\sum_{T\in\t_h}\sum_{F\in\f_T}h_F^{-1}\|&\vec{\pi}_{\boldsymbol{\cal P},T}^{k+1}(\boldsymbol{u}_{\mid T})-\boldsymbol{u}\|_F^2
		        \Bigg)^{\nicefrac{1}{2}}\\&\times\Bigg(\sum_{T\in\t_h}\sum_{F\in\f_T}h_F\|\qF-{\rm q}_{T\mid F}\|_F^2\Bigg)^{\nicefrac{1}{2}}.
          \end{split}
	\end{equation}
	Using, for all $T\in\t_h$, the approximation properties of $\vec{\pi}_{\boldsymbol{\cal P},T}^{k+1}$ on the faces of $T$ for the first factor on the right-hand side, and the triangle inequality along with a discrete trace inequality (see, e.g.,~\cite[Lemma 1.32]{Di-Pietro.Droniou:20}) for the second factor, we infer
	\begin{align}\label{eq:mh}
	  |\mh\big(\qh\big)|
          \lesssim \Bigg(
			\sum_{T\in\t_h}h_T^{2(k+1)}|\boldsymbol{u}|_{\boldsymbol H^{k+2}(T;\mathbb{R}^3)}^2
		\Bigg)^{\nicefrac{1}{2}}\|\qh\|_{{\rm Y},h}.
	\end{align}
	Plugging~\eqref{eq:lh} and~\eqref{eq:mh} into~\eqref{proof:standard.error.2} for $(\vh,\qh)$ such that $\|(\vh,\qh)\|_{\mathbb{Z},h}=1$ finally yields~\eqref{thm:standard.error.eq}.
\end{proof}

\subsubsection{Numerical results}\label{sec:standard:num}

Let the domain $\Omega$ be the unit cube $(0,1)^3$. We consider Problem~\eqref{eq:standard.strong} with exact solution
\begin{equation} \label{eq:sol}
\boldsymbol{u}(x_1,x_2,x_3) \defi \left(
\begin{tabular}{l}
$\cos(\pi x_2)\cos(\pi x_3)$
\\
$\cos(\pi x_1)\cos(\pi x_3)$
\\
$\cos(\pi x_1)\cos(\pi x_2)$
\end{tabular}
\right).
\end{equation}
Clearly, the function $\boldsymbol{u}$ satisfies~\eqref{eq:standard.strong:2}.
The source $\boldsymbol{f}$ is set according to~\eqref{eq:standard.strong:1}, and the zero tangential boundary condition~\eqref{eq:standard.strong:bc.u} is replaced by the non-homogeneous boundary condition stemming from~\eqref{eq:sol}.

We solve the discrete Problem~\eqref{eq:standard.discrete} with amended right-hand side accounting for the non-homogeneous boundary condition on two refined mesh sequences, of respectively cubic and regular tetrahedral meshes.
For each problem, the element unknowns for both the magnetic field and the Lagrange multiplier are locally eliminated using a Schur complement technique. This step is fully parallelizable. The resulting (condensed) global linear system is solved using the SparseLU direct solver of the Eigen library, on an Intel Xeon E-2176M 2.70GHz$\times$12 with 16GB of RAM (and up to 150GB of swap).
For $k\in\{0,1,2\}$, we depict on Figures~\ref{fig:standard.error.cubic} and~\ref{fig:standard.error.tetra.nostab}, respectively for the cubic and (regular) tetrahedral mesh families, the relative energy-error $\|\uh-\IntUh\boldsymbol{u}\|_{\boldsymbol{\rm X},h}/\|\IntUh\boldsymbol{u}\|_{\boldsymbol{\rm X},h}$ (top row) and $L^2$-error $\|\uTh-\boldsymbol{\pi}_{\boldsymbol{\cal P},h}^{k+1}\boldsymbol{u}\|_{\Omega}/\|\boldsymbol{\pi}_{\boldsymbol{\cal P},h}^{k+1}\boldsymbol{u}\|_{\Omega}$ (bottom row) as functions of
\begin{inparaenum}[(i)]
  \item the meshsize (left column),
  \item the solution time in seconds, i.e.~the time needed to solve the (condensed) global linear system (center column), and
  \item the number of (interface) degrees of freedom (DoF) (right column).
\end{inparaenum} 
For the two mesh families, we obtain, as predicted by Theorem~\ref{thm:standard.error}, an energy-error convergence rate of order $k+1$. We also observe a convergence rate of order $k+2$ for the $L^2$-error.
We remark that, whenever the solution is smooth enough (at least locally), raising the polynomial degree is computationally much more efficient than refining the mesh to increase the accuracy.
Following Remark~\ref{rk:tet.case.1}, we also solve on the (matching) tetrahedral mesh family a modifed version of Problem~\eqref{eq:standard.discrete} in which $\ch$ is removed, and we display on Figure~\ref{fig:standard.error.tetra.nostab} the results in dashed lines.
Also in this case, we obtain the predicted energy-error convergence rate of order $k+1$, and observe a convergence rate of order $k+2$ for the $L^2$-error.
We remark that the results using $\ch$ are always better (in absolute value) than those obtained without using it.

\begin{figure}\centering
  \ref{legend:hexa}
  \vspace{0.50cm}\\
  \begin{minipage}{0.30\textwidth}
    \ifMMMAS 
    \begin{tikzpicture}[scale=0.55]
    \else
    \begin{tikzpicture}[scale=0.60]
    \fi
      \begin{loglogaxis}[legend columns=-1, legend to name=legend:hexa]
        \addplot table[x=meshsize,y=errXnorm_u] {cv_field_cos/k0_hex.dat};
        \addplot table[x=meshsize,y=errXnorm_u] {cv_field_cos/k1_hex.dat};
        \addplot table[x=meshsize,y=errXnorm_u] {cv_field_cos/k2_hex.dat};
        \logLogSlopeTriangle{0.90}{0.4}{0.1}{1}{black};
        \logLogSlopeTriangle{0.90}{0.4}{0.1}{2}{black};
        \logLogSlopeTriangle{0.90}{0.4}{0.1}{3}{black};
        \legend{$k=0$,$k=1$,$k=2$};          
      \end{loglogaxis}
    \end{tikzpicture}
  \end{minipage}
  \hspace{0.025\textwidth}
  \begin{minipage}{0.30\textwidth}
    \ifMMMAS 
    \begin{tikzpicture}[scale=0.55]
    \else
    \begin{tikzpicture}[scale=0.60]
    \fi
      \begin{loglogaxis}
        \addplot table[x=tot_solution_time,y=errXnorm_u] {cv_field_cos/k0_hex.dat};
        \addplot table[x=tot_solution_time,y=errXnorm_u] {cv_field_cos/k1_hex.dat};
        \addplot table[x=tot_solution_time,y=errXnorm_u] {cv_field_cos/k2_hex.dat};
      \end{loglogaxis}
    \end{tikzpicture}
  \end{minipage}
  \hspace{0.025\textwidth} 
  \begin{minipage}{0.30\textwidth}
    \ifMMMAS 
    \begin{tikzpicture}[scale=0.55]
    \else
    \begin{tikzpicture}[scale=0.60]
    \fi
      \begin{loglogaxis}
        \addplot table[x=n_DOFs,y=errXnorm_u]{cv_field_cos/k0_hex.dat};
        \addplot table[x=n_DOFs,y=errXnorm_u]{cv_field_cos/k1_hex.dat};
        \addplot table[x=n_DOFs,y=errXnorm_u]{cv_field_cos/k2_hex.dat};
        \logLogSlopeTriangleNDOFs{0.10}{-0.4}{0.1}{1/3}{black};
        \logLogSlopeTriangleNDOFs{0.10}{-0.4}{0.1}{2/3}{black};
        \logLogSlopeTriangleNDOFs{0.10}{-0.4}{0.1}{1}{black};
      \end{loglogaxis}
    \end{tikzpicture}
  \end{minipage}
  \vspace{0.25cm}\\
  \begin{minipage}{0.30\textwidth}
    \ifMMMAS 
    \begin{tikzpicture}[scale=0.55]
    \else
    \begin{tikzpicture}[scale=0.60]
    \fi
      \begin{loglogaxis}
        \addplot table[x=meshsize,y=errL2_u] {cv_field_cos/k0_hex.dat};
        \addplot table[x=meshsize,y=errL2_u] {cv_field_cos/k1_hex.dat};
        \addplot table[x=meshsize,y=errL2_u] {cv_field_cos/k2_hex.dat};
        \logLogSlopeTriangle{0.90}{0.4}{0.1}{2}{black};
        \logLogSlopeTriangle{0.90}{0.4}{0.1}{3}{black};
        \logLogSlopeTriangle{0.90}{0.4}{0.1}{4}{black};
      \end{loglogaxis}
    \end{tikzpicture}
  \end{minipage}
  \hspace{0.025\textwidth}
  \begin{minipage}{0.30\textwidth}
     \ifMMMAS 
    \begin{tikzpicture}[scale=0.55]
    \else
    \begin{tikzpicture}[scale=0.60]
    \fi
      \begin{loglogaxis}
        \addplot table[x=tot_solution_time,y=errL2_u] {cv_field_cos/k0_hex.dat};
        \addplot table[x=tot_solution_time,y=errL2_u] {cv_field_cos/k1_hex.dat};
        \addplot table[x=tot_solution_time,y=errL2_u] {cv_field_cos/k2_hex.dat};
      \end{loglogaxis}
    \end{tikzpicture}
  \end{minipage}
  \hspace{0.025\textwidth}
  \begin{minipage}{0.30\textwidth}
    \ifMMMAS 
    \begin{tikzpicture}[scale=0.55]
    \else
    \begin{tikzpicture}[scale=0.60]
    \fi
      \begin{loglogaxis}
        \addplot table[x=n_DOFs,y=errL2_u]{cv_field_cos/k0_hex.dat};
        \addplot table[x=n_DOFs,y=errL2_u]{cv_field_cos/k1_hex.dat};
        \addplot table[x=n_DOFs,y=errL2_u]{cv_field_cos/k2_hex.dat};
        \logLogSlopeTriangleNDOFs{0.10}{-0.4}{0.1}{2/3}{black};
        \logLogSlopeTriangleNDOFs{0.10}{-0.4}{0.1}{1}{black};
        \logLogSlopeTriangleNDOFs{0.10}{-0.4}{0.1}{4/3}{black};
      \end{loglogaxis}
    \end{tikzpicture}
  \end{minipage}
  \caption{\label{fig:standard.error.cubic}
    Relative energy-error ${\|\uh-\IntXh\boldsymbol{u}\|_{\boldsymbol{\rm X},h}}/{\|\IntXh\boldsymbol{u}\|_{\boldsymbol{\rm X},h}}$ (\emph{top row}) and
    $L^2$-error ${\|\uTh-\boldsymbol{\pi}_{\boldsymbol{\cal P},h}^{k+1}\boldsymbol{u}\|_{\Omega}}/{\|\boldsymbol{\pi}_{\boldsymbol{\cal P},h}^{k+1}\boldsymbol{u}\|_{\Omega}}$ (\emph{bottom row})
    versus meshsize $h$ (\emph{left column}), solution time (\emph{center column}), and number of DoF (\emph{right column}) on cubic meshes for the test-case of Section~\ref{sec:standard:num}.}
\end{figure}

\begin{figure}\centering
  \ref{legend:tria}
  \vspace{0.50cm}\\
  \begin{minipage}[b]{0.30\columnwidth}
    \ifMMMAS 
    \begin{tikzpicture}[scale=0.55]
    \else
    \begin{tikzpicture}[scale=0.60]
    \fi
      \begin{loglogaxis}[legend columns=-1, legend to name=legend:tria]
        \addplot table[x=meshsize,y=errXnorm_u]{cv_field_cos/k0_tet_ch.dat};
        \addplot table[x=meshsize,y=errXnorm_u]{cv_field_cos/k1_tet_ch.dat};
        \addplot table[x=meshsize,y=errXnorm_u]{cv_field_cos/k2_tet_ch.dat};
        \addplot[dashed,color=blue,mark=+] table[x=meshsize,y=errXnorm_u] {cv_field_cos/k0_tet_nostab.dat};
        \addplot[dashed,color=red,mark=+] table[x=meshsize,y=errXnorm_u] {cv_field_cos/k1_tet_nostab.dat};
        \addplot[dashed,color=brown,mark=+] table[x=meshsize,y=errXnorm_u] {cv_field_cos/k2_tet_nostab.dat};
        \logLogSlopeTriangle{0.90}{0.4}{0.1}{1}{black};
        \logLogSlopeTriangle{0.90}{0.4}{0.1}{2}{black};
        \logLogSlopeTriangle{0.90}{0.4}{0.1}{3}{black};
        \legend{$k=0$,$k=1$,$k=2$};
      \end{loglogaxis}
    \end{tikzpicture}
  \end{minipage}
  \hspace{0.025\textwidth}
  \begin{minipage}[b]{0.30\textwidth}
    \ifMMMAS 
    \begin{tikzpicture}[scale=0.55]
    \else
    \begin{tikzpicture}[scale=0.60]
    \fi
      \begin{loglogaxis}
        \addplot table[x=tot_solution_time,y=errXnorm_u]{cv_field_cos/k0_tet_ch.dat};
        \addplot table[x=tot_solution_time,y=errXnorm_u]{cv_field_cos/k1_tet_ch.dat};
        \addplot table[x=tot_solution_time,y=errXnorm_u]{cv_field_cos/k2_tet_ch.dat};
        \addplot[dashed,color=blue,mark=+] table[x=tot_solution_time,y=errXnorm_u] {cv_field_cos/k0_tet_nostab.dat};
        \addplot[dashed,color=red,mark=+] table[x=tot_solution_time,y=errXnorm_u] {cv_field_cos/k1_tet_nostab.dat};
        \addplot[dashed,color=brown,mark=+] table[x=tot_solution_time,y=errXnorm_u] {cv_field_cos/k2_tet_nostab.dat};
      \end{loglogaxis}
    \end{tikzpicture}
  \end{minipage}
  \hspace{0.025\textwidth}
  \begin{minipage}[b]{0.30\textwidth}
    \ifMMMAS 
    \begin{tikzpicture}[scale=0.55]
    \else
    \begin{tikzpicture}[scale=0.60]
    \fi
      \begin{loglogaxis}
        \addplot table[x=n_DOFs,y=errXnorm_u]{cv_field_cos/k0_tet_ch.dat};
        \addplot table[x=n_DOFs,y=errXnorm_u]{cv_field_cos/k1_tet_ch.dat};
        \addplot table[x=n_DOFs,y=errXnorm_u]{cv_field_cos/k2_tet_ch.dat};
        \addplot[dashed,color=blue,mark=+] table[x=n_DOFs,y=errXnorm_u]{cv_field_cos/k0_tet_nostab.dat};
        \addplot[dashed,color=red,mark=+] table[x=n_DOFs,y=errXnorm_u]{cv_field_cos/k1_tet_nostab.dat};
        \addplot[dashed,color=brown,mark=+] table[x=n_DOFs,y=errXnorm_u]{cv_field_cos/k2_tet_nostab.dat};
        \logLogSlopeTriangleNDOFs{0.10}{-0.4}{0.1}{1/3}{black};
        \logLogSlopeTriangleNDOFs{0.10}{-0.4}{0.1}{2/3}{black};
        \logLogSlopeTriangleNDOFs{0.10}{-0.4}{0.1}{1}{black};
      \end{loglogaxis}
    \end{tikzpicture}
  \end{minipage}
  \vspace{0.25cm}\\
  \begin{minipage}[b]{0.30\textwidth}
    \ifMMMAS 
    \begin{tikzpicture}[scale=0.55]
    \else
    \begin{tikzpicture}[scale=0.60]
    \fi
      \begin{loglogaxis}    
        \addplot table[x=meshsize,y=errL2_u]{cv_field_cos/k0_tet_ch.dat};
        \addplot table[x=meshsize,y=errL2_u]{cv_field_cos/k1_tet_ch.dat};
        \addplot table[x=meshsize,y=errL2_u]{cv_field_cos/k2_tet_ch.dat};
        \addplot[dashed,color=blue,mark=+] table[x=meshsize,y=errL2_u] {cv_field_cos/k0_tet_nostab.dat};
        \addplot[dashed,color=red,mark=+] table[x=meshsize,y=errL2_u] {cv_field_cos/k1_tet_nostab.dat};
        \addplot[dashed,color=brown,mark=+] table[x=meshsize,y=errL2_u] {cv_field_cos/k2_tet_nostab.dat};    
        \logLogSlopeTriangle{0.90}{0.4}{0.1}{2}{black};
        \logLogSlopeTriangle{0.90}{0.4}{0.1}{3}{black};
        \logLogSlopeTriangle{0.90}{0.4}{0.1}{4}{black};
      \end{loglogaxis}
    \end{tikzpicture}
  \end{minipage}
  \hspace{0.025\textwidth}
  \begin{minipage}[b]{0.30\textwidth}
    \ifMMMAS 
    \begin{tikzpicture}[scale=0.55]
    \else
    \begin{tikzpicture}[scale=0.60]
    \fi
      \begin{loglogaxis}
        \addplot table[x=tot_solution_time,y=errL2_u]{cv_field_cos/k0_tet_ch.dat};
        \addplot table[x=tot_solution_time,y=errL2_u]{cv_field_cos/k1_tet_ch.dat};
        \addplot table[x=tot_solution_time,y=errL2_u]{cv_field_cos/k2_tet_ch.dat};
        \addplot[dashed,color=blue,mark=+] table[x=tot_solution_time,y=errL2_u] {cv_field_cos/k0_tet_nostab.dat};
        \addplot[dashed,color=red,mark=+] table[x=tot_solution_time,y=errL2_u] {cv_field_cos/k1_tet_nostab.dat};
        \addplot[dashed,color=brown,mark=+] table[x=tot_solution_time,y=errL2_u] {cv_field_cos/k2_tet_nostab.dat};
      \end{loglogaxis}
    \end{tikzpicture}
  \end{minipage}
  \hspace{0.025\textwidth}
  \begin{minipage}[b]{0.30\textwidth}
    \ifMMMAS 
    \begin{tikzpicture}[scale=0.55]
    \else
    \begin{tikzpicture}[scale=0.60]
    \fi
      \begin{loglogaxis}
        \addplot table[x=n_DOFs,y=errL2_u]{cv_field_cos/k0_tet_ch.dat};
        \addplot table[x=n_DOFs,y=errL2_u]{cv_field_cos/k1_tet_ch.dat};
        \addplot table[x=n_DOFs,y=errL2_u]{cv_field_cos/k2_tet_ch.dat};
        \addplot[dashed,color=blue,mark=+] table[x=n_DOFs,y=errL2_u]{cv_field_cos/k0_tet_nostab.dat};
        \addplot[dashed,color=red,mark=+] table[x=n_DOFs,y=errL2_u]{cv_field_cos/k1_tet_nostab.dat};
        \addplot[dashed,color=brown,mark=+] table[x=n_DOFs,y=errL2_u]{cv_field_cos/k2_tet_nostab.dat};
        \logLogSlopeTriangleNDOFs{0.10}{-0.4}{0.1}{2/3}{black};
        \logLogSlopeTriangleNDOFs{0.10}{-0.4}{0.1}{1}{black};
        \logLogSlopeTriangleNDOFs{0.10}{-0.4}{0.1}{4/3}{black};
      \end{loglogaxis}
    \end{tikzpicture}
  \end{minipage}
  \caption{\label{fig:standard.error.tetra.nostab}
    Relative energy-error ${\|\uh-\IntXh\boldsymbol{u}\|_{\boldsymbol{\rm X},h}}/{\|\IntXh\boldsymbol{u}\|_{\boldsymbol{\rm X},h}}$ (\emph{top row}) and
    $L^2$-error ${\|\uTh-\boldsymbol{\pi}_{\boldsymbol{\cal P},h}^{k+1}\boldsymbol{u}\|_{\Omega}}/{\|\boldsymbol{\pi}_{\boldsymbol{\cal P},h}^{k+1}\boldsymbol{u}\|_{\Omega}}$ (\emph{bottom row})
    versus meshsize $h$ (\emph{left column}), solution time (\emph{center column}), and number of DoF (\emph{right column}) on tetrahedral meshes for the test-case of Section~\ref{sec:standard:num}, and comparison (dashed lines) with the case where the stabilization bilinear form $\ch$ is removed.}
\end{figure}

\subsection{Vector potential formulation}\label{sec:general}

\subsubsection{The model}

The (second-order) vector potential formulation of the magnetostatics problem consists, in its generalized form, in finding the magnetic vector potential $\boldsymbol{u}:\Omega\rightarrow\Real^3$ and the Lagrange multiplier $p:\Omega\rightarrow\Real$ such that
\begin{subequations}
  \label{eq:strong}
  \begin{alignat}{2}
    \Curl\!\big(\!\Curl \boldsymbol{u}\big) + \Grad p &= \boldsymbol{f} &\qquad&\text{in $\Omega$},\label{eq:strong:1}
    \\
    \Div \boldsymbol{u} &= 0 &\qquad&\text{in $\Omega$},\label{eq:strong:2}
	\\
    \normal{\times}(\boldsymbol{u}{\times}\normal)  &=  \boldsymbol{0} &\qquad&\text{on $\partial \Omega$},\label{eq:strong:bc.u}
	\\
    p  &=  0 &\qquad&\text{on $\partial \Omega$},\label{eq:strong:bc.p}
  \end{alignat}
\end{subequations}
where the current density $\boldsymbol{f}:\Omega\to\Real^3$ is no longer assumed to be divergence-free, whence the introduction of the Lagrange multiplier term in~\eqref{eq:strong:1}. When $\Div\vec{f}=0$ in $\Omega$, $p=0$ in $\Omega$ and, letting $\vec{b}\defi\Curl\vec{u}$, Problem~\eqref{eq:strong} is then equivalent, in the class of vector potentials satisfying the Coulomb gauge, to the following problem (cf.~Remark~\ref{rem:top}):
\begin{equation} \label{eq:strong.equiv}
  \Curl \boldsymbol{b} = \boldsymbol{f} \quad\text{in $\Omega$},\qquad\Div \boldsymbol{b} = 0 \quad\text{in $\Omega$},\qquad\boldsymbol{b}{\cdot}\normal  =  0 \quad\text{on $\partial \Omega$}.
\end{equation}
The function $\vec{u}$ is then the vector potential associated to the magnetic induction $\vec{b}$.
Assuming that $\boldsymbol{f}\in \boldsymbol L^2(\Omega;\mathbb{R}^3)$, we consider the following equivalent weak formulation of Problem~\eqref{eq:strong}: Find $(\boldsymbol{u},p)\in\Hzcurl\times H^1_0(\Omega)$ such that
\begin{subequations}\label{eq:weak}
	\begin{alignat}{2}\label{eq:weak.1}
		a(\boldsymbol{u},\boldsymbol{v}) + b(\boldsymbol{v},p) & = (\boldsymbol{f},\boldsymbol{v})_\Omega &\qquad&\forall \boldsymbol{v}\in\Hzcurl,
		\\\label{eq:weak.2}
		b(\boldsymbol{u},q) & = 0 &\qquad&\forall q\in H^1_0(\Omega),
	\end{alignat}
\end{subequations}
where the bilinear forms $a:\Hcurl\times\Hcurl\rightarrow\Real$ and $b:\Hcurl\times H^1(\Omega) \rightarrow\Real$ are defined in~\eqref{def:standard.bf.c0}.
Using the decomposition~\eqref{eq:helmholtz}, Problem~\eqref{eq:weak} can be equivalently rewritten under the following fully decoupled form: Find $\vec{u}\in\Hzcurl\cap\Hdivz$ and $p\in H^1_0(\Omega)$ such that
\begin{alignat*}{2}
  a(\vec{u},\vec{\eta})&=(\vec{f},\vec{\eta})_{\Omega}&\qquad&\forall\;\vec{\eta}\in\Hzcurl\cap\Hdivz,\\b(\Grad\xi,p)&=(\vec{f},\Grad\xi)_{\Omega}&\qquad&\forall\;\xi\in H^1_0(\Omega),
\end{alignat*}
whose well-posedness directly follows from the first Weber inequality~\eqref{eq:weber} and from the Lax--Milgram lemma.

\subsubsection{The HHO method} 

We consider the hybrid spaces introduced in Section~\ref{sse:hybrid} (the space $\Xh$ for the magnetic vector potential, and $\Yh$ for the Lagrange multiplier), up to a slight modification of the space $\Xh$ defined in~\eqref{def:DOFs.Xh}.
To this end, we introduce, for any $q\in\mathbb{N}$ and any $F\in\f_h$, the space
\begin{equation}\label{spa}
  {\boldsymbol{{\cal P}}}^{q}_\flat(F)\defi\boldsymbol\Poly^{q-1}(F)\oplus\Grad_\tau\!\big(\widetilde{\Poly}^{q+1}(F)\big),
\end{equation}
with $\widetilde{\Poly}^{q+1}(F)$ denoting the space of homogeneous scalar-valued polynomials of total degree $q+1$ on $F$, and the convention that ${\boldsymbol{{\cal P}}}^{-1}(F)\defi\{\boldsymbol{0}\}$. 
Consistently with our notation so far, we let $\vec{\pi}_{{\boldsymbol{{\cal P}}}_\flat,F}^{q}$ denote the $\boldsymbol L^2(F;\mathbb{R}^2)$-orthogonal projector onto ${\boldsymbol{{\cal P}}}_\flat^{q}(F)$.
With this new space at hand, we define
\begin{equation}\label{def:DOFs.Xh.sharp}
  \Xhs\defi\left\{ \vh = \big((\vT)_{T\in\t_h},(\vF)_{F\in\f_h}\big) \st 
  \begin{alignedat}{2}
    \vT&\in \PkpdT &\quad& \forall T\in\t_h
    \\
    \vF&\in\boldsymbol{{\cal P}}^{k+1}_\flat(F) &\quad& \forall F\in\f_h
  \end{alignedat}
  \right\},
\end{equation}
which is from now on meant to replace the space $\Xh$. The space $\Yh$ keeps the same definition~\eqref{def:DOFs.Yh}. We also introduce the spaces $\Xhzs$ and $\XXhzs$, that are respectively obtained through definitions~\eqref{def:DOFs.Xhz} and~\eqref{def:XXhz}, up to the replacement therein of $\Xh$ by $\Xhs$, and of $\Xhz$ by $\Xhzs$.
In turn, the interpolator $\IntXhs: \boldsymbol H^1(\Omega;\mathbb{R}^3)\rightarrow\Xhs$ is defined as in formula~\eqref{def:Icurl.h}, up to the replacement of the projector $\vec{\pi}_{\boldsymbol{{\cal G}},F}^{k+1}$ by $\vec{\pi}_{\boldsymbol{{\cal P}}_\flat,F}^{k+1}$.
Here, and as opposed to Section~\ref{sse:field} (cf.~Remark~\ref{rem:curl}), because of the fact that we will have to reconstruct a discrete ${\bf curl}$ operator, we need to consider a space for the vectorial face unknowns that contains $\boldsymbol{{\cal P}}^k(F)$ (this is indeed needed to prove optimal approximation properties for the ${\bf curl}$ reconstruction operator). Since, for stability reasons, the space for face unknowns must also contain $\boldsymbol{{\cal G}}^{k+1}(F)$, we consider the sum of these two spaces, which rewrites as the direct sum~\eqref{spa}. Notice that ${\boldsymbol{{\cal P}}}^{k+1}_\flat(F)$ is strictly sandwiched between $\boldsymbol{{\cal P}}^k(F)$ and $\boldsymbol{{\cal P}}^{k+1}(F)$.
\begin{remark}[Validity of the results of Section \ref{sse:dwi}]\label{rem:latitude}
  For all $F\in\f_h$, we have
  $$\boldsymbol{{\cal G}}^{k+1}(F)\subseteq\boldsymbol{{\cal P}}^{k+1}_\flat(F),$$
  and it can be checked that, up to the replacement of the $\boldsymbol{L}^2(F;\mathbb{R}^2)$-orthogonal projector $\vec{\pi}^{k+1}_{\boldsymbol{{\cal G}},F}$ onto $\boldsymbol{{\cal G}}^{k+1}(F)$ by the projector $\vec{\pi}^{k+1}_{\boldsymbol{{\cal P}}_\flat,F}$ onto $\boldsymbol{{\cal P}}^{k+1}_\flat(F)$, all the results in Section~\ref{sse:dwi} remain valid when $\Xh$ is replaced by $\Xhs$ as defined in~\eqref{def:DOFs.Xh.sharp}, including the discrete Weber inequality of Theorem~\ref{thm:max.ineq} and its generalization of Corollary~\ref{cor:max.ineq} (observe, in particular, that the crucial estimates~\eqref{eq:crucial}--\eqref{eq:max.proof:I2b} still hold true under these changes).
\end{remark}
We define the discrete bilinear forms $\ah:\Xhs\times\Xhs\rightarrow\Real$, $\bh:\Xhs\times\Yh\rightarrow\Real$, and $\dh:\Yh\times\Yh\rightarrow\Real$ such that
\begin{subequations}\label{def:general.bf.h}
	\begin{alignat}{2}\label{def:general.bf.ah}
		\ah(\wh,\vh) & \defi \big(\Ckh\wh,\Ckh\vh\big)_\Omega + \sh(\wh,\vh),
		\\\label{def:general.bf.bh}
		\bh(\wh,\qh) & \defi \big(\wTh,\Gkh\qh\big)_\Omega,
		\\\label{def:general.bf.ch}
		\dh(\rh,\qh) & \defi \sum_{T\in\t_h} \sum_{F\in\f_T} h_F \big(\rF-{\rm r}_{T\mid F},\qF-{\rm q}_{T\mid F}\big)_F ,
	\end{alignat}
\end{subequations}
where $\Gkh:\Yh\rightarrow\PkpdTh$ is the gradient reconstruction operator introduced in Section~\ref{sec:grad.recons}, and $\sh:\Xhs\times\Xhs\rightarrow\Real$ is the stabilization bilinear form such that
\begin{equation}\label{def:general.bf.sh}
  \sh(\wh,\vh)\defi \sum_{T\in\t_h}\sum_{F\in\f_T} h_F^{-1}\big(\vec{\pi}_{\boldsymbol{{\cal P}}_\flat,F}^{k+1}\big(\vec{\gamma}_{\tau,F}(\wT)-\wF\big),\vec{\pi}_{\boldsymbol{{\cal P}}_\flat,F}^{k+1}\big(\vec{\gamma}_{\tau,F}(\vT)-\vF\big)\big)_F.
\end{equation}
In~\eqref{def:general.bf.ah}, $\Ckh:\Xhs\rightarrow\boldsymbol{{\cal R}}^k({\t_h})$ (with $\boldsymbol{\cal R}^k({\t_h})$ defined in Section~\ref{sse:hybrid}) is the global discrete ${\bf curl}$ reconstruction operator such that its local restriction $\CkT:\XTs\rightarrow\boldsymbol{{\cal R}}^k(T)$ to any $T\in\t_h$ solves the following well-posed problem: For all $\vTF\in\XTs$,
\begin{equation}\label{def:CTk}
  \big(\CkT\vTF, \vec{w}\big)_T = (\vT,\Curl \vec{w})_T
  + \sum_{F\in\f_T}\big(\vF,\vec{\gamma}_{\tau,F}(\vec{w}{\times}\normal_{TF})\big)_F
  \quad\forall\vec{w}\in\boldsymbol{{\cal R}}^k(T).
\end{equation}
With this definition at hand, one can prove the following commutation property.
\begin{lemma}[Commutation property]\label{lemma:commut.CTkITk}
	For all $\boldsymbol{v}\in \boldsymbol H^1(\Omega;\mathbb{R}^3)$, we have
	\begin{equation}\label{prop:commut.CTkITk}
		\Ckh (\IntXhs \boldsymbol{v}) = \vec{\pi}_{\boldsymbol{{\cal R}},h}^k (\Curl \boldsymbol{v}),
	\end{equation}
        where we remind the reader that $\vec{\pi}_{\boldsymbol{\cal R},h}^k$ is the $\boldsymbol L^2(\Omega;\mathbb{R}^3)$-orthogonal projector onto $\boldsymbol{\cal R}^k({\t_h})$.
\end{lemma}

\begin{proof}
  For $\boldsymbol{v}\in \boldsymbol H^1(\Omega;\mathbb{R}^3)$, let $\vh \defi \IntXhs \boldsymbol{v}$. Then, for any $T\in\t_h$,
  $$\vTF=\left(\vec{\pi}_{\boldsymbol{\cal P},T}^{k+1}(\boldsymbol{v}_{\mid T}),\big(\vec{\pi}_{\boldsymbol{{\cal P}}_\flat,F}^{k+1}\big(\vec{\gamma}_{\tau,F}(\boldsymbol{v})\big)\big)_{F\in\f_T}\right).$$
  Plugging $\vTF$ into~\eqref{def:CTk}, and using that $\Curl \vec{w}\in\boldsymbol\Poly^{k-1}(T)\subset\PkpdT$ and that $\vec{\gamma}_{\tau,F}(\vec{w}{\times}\normal_{TF})\in\PkdF\subset\boldsymbol{{\cal P}}^{k+1}_\flat(F)$ for all $F\in\f_T$ to remove the projectors, one gets, for all $\vec{w}\in\boldsymbol{{\cal R}}^k(T)$,
	\begin{equation}\label{proof:commut.CTkITK:1}
	  \big(\CkT\vTF, \vec{w}\big)_T = (\vec{v},\Curl \vec{w})_T
          + \sum_{F\in\f_T}
            \big(\vec{v}_{\mid F},\vec{w}_{\mid F}{\times}\normal_{TF}\big)_F.
	\end{equation}
	Integrating by parts the right-hand side of \eqref{proof:commut.CTkITK:1}, we readily infer \eqref{prop:commut.CTkITk}.
\end{proof}
\begin{remark}[Variant on $\Ckh$] \label{rem:variant}
  An alternative choice consists in reconstructing the discrete $\Curl$ in $\PkdTh$, which requires to solve larger local problems for $k\geq 1$ (for example, the space $\boldsymbol{{\cal R}}^2(T)$ has dimension $26$, whereas $\boldsymbol\Poly^2(T)$ has dimension $30$). In this case, the commutation property~\eqref{prop:commut.CTkITk} reads $\Ckh(\IntXhs \boldsymbol{v}) = \vec{\pi}_{\boldsymbol{\cal P},h}^{k} (\Curl \boldsymbol{v})$. This is the approach pursued in~\ifMMMAS\else\cite{Nguyen.ea:11} and~\cite{Chen.ea:17}\fi.
  The numerical tests we have performed (not reported here) indicate that, interestingly, reconstructing the discrete curl in $\PkdTh$ instead of $\boldsymbol{\cal R}^k({\t_h})$, besides being computationally more expensive, sometimes deteriorates the accuracy of the approximation.
\end{remark}
The HHO method for Problem~\eqref{eq:weak} reads: Find $(\uh,\ph)\in\Xhzs\times\Yhz$ such that
\begin{subequations}\label{eq:general.discrete}
	\begin{alignat}{3}\label{eq:general.discrete.1}
		\ah(\uh,\vh) + \bh(\vh,\ph) & = (\boldsymbol{f},\vTh)_\Omega
		&\qquad&\forall \vh\in\Xhzs,
		\\\label{eq:general.discrete.2}
		-\bh(\uh,\qh) +\dh(\ph,\qh)& = 0 &\qquad&\forall \qh\in \Yhz.
	\end{alignat}
\end{subequations}
Notice that, as opposed to $\boldsymbol{u}$ and $p$ in Problem~\eqref{eq:weak}, one cannot efficiently solve for $\uh$ and $\ph$ independently in Problem~\eqref{eq:general.discrete} as the $\Curl$-$\Grad$ orthogonality is lost at the discrete level.
\begin{remark}[Divergence-free current density]
  At the continuous level, remark that $\Div\vec{f}=0$ in $\Omega$ implies $p = 0$ in $\Omega$. At the discrete level, when $\vec{f}$ is divergence-free, one can hence mimick the strategy advocated in Section~\ref{sse:field} and replace in~\eqref{eq:general.discrete.2} the bilinear form $\dh$ by the (consistent) bilinear form $\ch$ given by~\eqref{def:standard.bf.ch}. Doing so improves the stability of the method without deteriorating its convergence properties (cf.~Remark~\ref{rem:stab}). 
\end{remark}

Letting $\Zhs\defi\Xhs\times\Yh$ and $\Zhzs\defi\Xhzs\times\Yhz$, we notice that Problem~\eqref{eq:general.discrete} can be equivalently rewritten: Find $(\uh,\ph)\in\Zhzs$ such that
\begin{equation}\label{eq:general.discrete.equiv}
  \Ah\big((\uh,\ph),(\vh,\qh)\big)=(\boldsymbol{f},\vTh)_\Omega\qquad\forall\,(\vh,\qh)\in\Zhzs,
\end{equation}
where the bilinear form $\Ah:\Zhs\times\Zhs\to\Real$ is defined by
\begin{equation}\label{def:general.Ah}
	\Ah\big((\wh,\rh),(\vh,\qh)\big) \defi \ah(\wh,\vh) + \bh(\vh,\rh) - \bh(\wh,\qh) + \dh(\rh,\qh).
\end{equation}
We also let, in analogy with \eqref{def:ZZhz},
\begin{equation} \label{def:general.ZZhz}
  \begin{split}
    \ZZhzs\defi&\left\{
    (\wh,\rh)\in\Zhzs\,:\,-\bh(\wh,\qh) +\dh(\rh,\qh) = 0 \quad\forall\qh\in\Yhz
    \right\}
    \\
    =&\left\{(\wh,\rh)\in\Zhzs\,:\,\Ah\big((\wh,\rh),(\underline{\vec{0}}_h,\qh)\big) = 0 \quad\forall\qh\in\Yhz\right\}.
  \end{split}
\end{equation}

\subsubsection{Stability analysis}

We equip the spaces $\Xhs$ and $\Yh$ with the seminorms
\begin{subequations}\label{def:norms.general}
  \begin{alignat}{1}\label{def:norms.general:curl}
    \|\wh\|_{\boldsymbol{{\rm X}},\flat,h} &\defi \left(
    \|\Curl_h\!\wTh \|_\Omega^2 + \sh(\wh,\wh)
    \right)^{\nicefrac12},    
    \\\label{def:norms.general:grad}
    \|\rh\|_{{\rm Y},\flat,h} &\defi\left(
    \sum_{T\in\t_h}h_T^2\|\Grad{\rm r}_T\|_T^2 + \dh(\rh,\rh)
    \right)^{\nicefrac12}.
  \end{alignat}
\end{subequations}
One can easily verify that $\|\cdot\|_{{\rm Y},\flat,h}$ defines a norm on $\Yhz$.
We now equip $\Zhs$ with the seminorm
\begin{equation}\label{def:general.norm}
  \|(\wh,\rh)\|_{\mathbb{Z},\flat,h}\defi\left(
  \|\wh\|_{\boldsymbol{{\rm X}},\flat,h}^2+\|\rh\|_{{\rm Y},\flat,h}^2
  \right)^{\nicefrac12}.
\end{equation}
\begin{lemma}[Norm $\|\cdot\|_{\mathbb{Z},\flat,h}$] \label{le:norm.Z.general}
	The map $\|\cdot\|_{\mathbb{Z},\flat,h}$ defines a norm on $\ZZhzs$ defined by~\eqref{def:general.ZZhz}.
\end{lemma}
\begin{proof}
  The seminorm property being straightforward, we only need to prove that, for all couples $(\wh,\rh)\in\ZZhzs$, $\|(\wh,\rh)\|_{\mathbb{Z},\flat,h} = 0 $ implies $(\wh,\rh)=(\underline{\boldsymbol{0}}_h,\underline{0}_h)$.
	Let then $(\wh,\rh)\in\ZZhzs$ be such that $\|(\wh,\rh)\|_{\mathbb{Z},\flat,h} = 0 $. We infer that $\|\wh\|_{\boldsymbol{{\rm X}},\flat,h}=0$ and $\|\rh\|_{{\rm Y},\flat,h}=0$.
        Since $\|\cdot\|_{{\rm Y},\flat,h}$ is a norm on $\Yhz$, we directly get from the second relation that $\rh=\underline{0}_h$.
        Now, owing to the definitions~\eqref{def:general.ZZhz} of $\ZZhzs$,~\eqref{def:general.bf.bh} of $\mathrm{b}_h$, and to the fact that $\XXhzs$ is defined as in~\eqref{def:XXhz} with $\Xhz$ replaced by $\Xhzs$, we infer from $(\wh,\rh)\in\ZZhzs$ and $\rh=\underline{0}_h$ that $\wh\in\XXhzs$.
        By Corollary \ref{co:norm.X} and Remark~\ref{rem:latitude}, $\|\cdot\|_{\boldsymbol{{\rm X}},\flat,h}$ defines a norm on $\XXhzs$, hence $\wh=\underline{\vec{0}}_h$, which concludes the proof.
\end{proof}

We now state some preliminary results for the stability analysis of Problem~\eqref{eq:general.discrete}.
\begin{lemma}[Equivalences of seminorms]\label{lem:stab:ah}
  The following holds true:
  \begin{subequations}\label{lem:stab:ah.eq}
    \begin{gather}\label{lem:stab:ah.eq1}
      \|\wh\|_{\boldsymbol{{\rm X}},\flat,h}^2 \lesssim \ah(\wh,\wh) \lesssim \|\wh\|_{\boldsymbol{{\rm X}},\flat,h}^2
      \qquad\forall\wh\in\Xhs,
      \\ \label{lem:stab:ah.eq2}
      \|\rh\|_{{\rm Y},\flat,h}^2\lesssim \sum_{T\in\t_h}h_T^2\|\GkT\rTF\|_T^2+\dh(\rh,\rh) \lesssim \|\rh\|_{{\rm Y},\flat,h}^2
      \qquad\forall\rh\in\Yh.
    \end{gather}
  \end{subequations}
\end{lemma}

\begin{proof}
  Let us prove~\eqref{lem:stab:ah.eq1}. Let $\wh\in\Xhs$, and $T\in\t_h$.
  By the definition~\eqref{def:CTk} of $\CkT$, testing with $\vec{w}=\Curl\wT\in\boldsymbol{{\cal R}}^k(T)$, integrating by parts, and using the fact that $\vec{\gamma}_{\tau,F}\big(\Curl\wT{\times}\normal_{TF}\big)\in\PkdF\subset\boldsymbol{{\cal P}}^{k+1}_\flat(F)$, we infer
  \begin{multline*}
    \|\Curl\wT\|_T^2
    =\big(\CkT\wTF,\Curl\wT\big)_T
    \\+\sum_{F\in\f_T}\big(\vec{\pi}^{k+1}_{\boldsymbol{{\cal P}}_\flat,F}\big(\vec{\gamma}_{\tau,F}(\wT)-\wF\big),\vec{\gamma}_{\tau,F}(\Curl\wT{\times}\normal_{TF})\big)_F.
  \end{multline*}
  By the Cauchy--Schwarz inequality, a discrete trace inequality (see, e.g.,~\cite[Lemma 1.32]{Di-Pietro.Droniou:20}), and recalling the definition~\eqref{def:general.bf.ah} of $\ah$, we get $\|\Curl_h\!\wTh\|_{\Omega}^2\lesssim\ah(\wh,\wh)$,
  and the first inequality in~\eqref{lem:stab:ah.eq1} follows by adding $\sh(\wh,\wh)$ to both sides.
  To prove the second inequality, we test~\eqref{def:CTk} with $\vec{w}=\CkT\wTF\in\boldsymbol{{\cal R}}^k(T)$ to infer
  \begin{multline*}
    \|\CkT\wTF\|_T^2=\big(\Curl\wT,\CkT\wTF\big)_T
    \\-\sum_{F\in\f_T}\big(\vec{\pi}^{k+1}_{\boldsymbol{{\cal P}}_\flat,F}\big(\vec{\gamma}_{\tau,F}(\wT)-\wF\big),\vec{\gamma}_{\tau,F}(\CkT\wTF{\times}\normal_{TF})\big)_F,
  \end{multline*}
  and we conclude by the same arguments.
  The proof of~\eqref{lem:stab:ah.eq2} is similar and is omitted for brevity.
\end{proof}

\begin{lemma}[Control of $\Gkh$]\label{lem:infsupbh}
  For all $(\wh,\rh)\in\Zhs$, there exists $\vhs\in\Xhzs$ satisfying
  \[
  \|\vThs\|_{\Omega}^2+\|\vhs\|_{\boldsymbol{{\rm X}},\flat,h}^2\lesssim\sum_{T\in\t_h}h_T^2\|\GkT\rTF\|_T^2,
  \]
  such that the following holds true:
  \begin{equation}\label{lem:infsupbh.eq}
    \sum_{T\in\t_h}h_T^2\|\GkT\rTF\|_T^2
    \lesssim
    \Ah\big((\wh,\rh),(\vhs,\underline{0}_h)\big)+\|\wh\|_{\boldsymbol{{\rm X}},\flat,h}^2.
    \end{equation}
\end{lemma}

\begin{proof}
	Let $(\wh,\rh)\in\Zhs$. We define $\vhs\in\Xhzs$ such that
	\begin{equation}\label{proof:infsupbh.1}
            \vTs \defi h_T^2\GkT\rTF\quad\forall T\in\t_h,\qquad\vFs \defi \boldsymbol{0}\quad\forall F\in\f_h.
	\end{equation}
        We immediately verify, since $h_T\leq{\rm diam}(\Omega)$ for all $T\in\t_h$, that the following holds true: $\|\vThs\|_{\Omega}^2\lesssim\sum_{T\in\t_h}h_T^2\|\GkT\rTF\|_T^2$.
	From the definitions~\eqref{proof:infsupbh.1} of $\vhs$,~\eqref{def:general.bf.bh} of $\bh$, and~\eqref{def:general.Ah} of $\Ah$, we infer
	\begin{equation*}
		\sum_{T\in\t_h}h_T^2\|\GkT\rTF\|_T^2 
		=
		\bh(\vhs,\rh)=\Ah\big((\wh,\rh),(\vhs,\underline{0}_h)\big)-\ah(\wh,\vhs).
	\end{equation*}
        The Cauchy--Schwarz inequality followed by the second inequality in~\eqref{lem:stab:ah.eq1} then yields
        \begin{equation}\label{proof:infsupbh.2}
          \sum_{T\in\t_h}h_T^2\|\GkT\rTF\|_T^2 
		\lesssim\Ah\big((\wh,\rh),(\vhs,\underline{0}_h)\big)+\|\wh\|_{\boldsymbol{{\rm X}},\flat,h}\|\vhs\|_{\boldsymbol{{\rm X}},\flat,h}.
        \end{equation}
	Using the definitions~\eqref{proof:infsupbh.1} of $\vhs$ and~\eqref{def:norms.general:curl} of $\|\cdot\|_{\boldsymbol{{\rm X}},\flat,h}$, we get
	\begin{equation} \label{proof:infsupbh.3}
          \begin{split}
	    \|\vhs\|_{\boldsymbol{{\rm X}},\flat,h}^2 &= \sum_{T\in\t_h}\left(
	    \|\Curl\vTs\|_T^2 + \sum_{F\in\f_T}h_F^{-1}\|\vec{\pi}_{\boldsymbol{{\cal P}}_\flat,F}^{k+1}\big(\vec{\gamma}_{\tau,F}(\vTs)\big)\|_F^2
	    \right)
	    \\&\lesssim \sum_{T\in\t_h}h_T^{-2}\|\vTs\|_T^2=\sum_{T\in\t_h}h_T^2\|\GkT\rTF\|_T^2,
	  \end{split}
        \end{equation}
	where we have used the $\boldsymbol L^2(F;\mathbb{R}^2)$-boundedness of $\vec{\pi}_{\boldsymbol{{\cal P}}_\flat,F}^{k+1}$, as well as an inverse inequality together with a discrete trace inequality (see, e.g.,~\cite[Lemmas 1.28 and 1.32]{Di-Pietro.Droniou:20}).
	Starting from~\eqref{proof:infsupbh.2}, and using~\eqref{proof:infsupbh.3} combined with a Young inequality for the last term in the right-hand side eventually yields the expected result~\eqref{lem:infsupbh.eq}.
\end{proof}

We are now in position to show well-posedness for Problem~\eqref{eq:general.discrete}.
\begin{lemma}[Well-posedness]\label{thm:wellposedness}
  For all $\zh\in\Zhs$, there exists $\vhs\in\Xhzs$ satisfying $\|\vThs\|_{\Omega}+\|\vhs\|_{\boldsymbol{{\rm X}},\flat,h}\lesssim\|\zh\|_{\mathbb{Z},\flat,h}$ and such that 
  \begin{equation}\label{thm:wellposedness.eq}
    \Ah\big(\zh,\zh\big)+\Ah\big(\zh,(\vhs,\underline{0}_h)\big)\gtrsim\|\zh\|_{\mathbb{Z},\flat,h}^2.
  \end{equation}
  Hence, Problem~\eqref{eq:general.discrete} is well-posed, and the following a priori bound holds true:
  \begin{equation} \label{eq:apr}
    \|(\uh,\ph)\|_{\mathbb{Z},\flat,h}\lesssim\|\vec{f}\|_{\Omega}.
  \end{equation}
\end{lemma}

\begin{proof}
  Let $\zh=(\wh,\rh)\in\Zhs$.
	By the first inequality in~\eqref{lem:stab:ah.eq1} and~\eqref{def:general.Ah}, one has
	\begin{equation} \label{eq:eqfirst}
		\|\wh\|_{\boldsymbol{{\rm X}},\flat,h}^2 + \dh(\rh,\rh) \lesssim \ah(\wh,\wh) + \dh(\rh,\rh) =\Ah\big(\zh,\zh\big).
	\end{equation}
	By Lemma~\ref{lem:infsupbh} combined with~\eqref{eq:eqfirst}, one also has
        \begin{equation} \label{eq:eqsecond}
          \sum_{T\in\t_h}h_T^2\|\GkT\rTF\|_T^2
		\lesssim
		\Ah\big(\zh,(\vhs,\underline{0}_h)\big)+\Ah\big(\zh,\zh\big),
        \end{equation}
        for some $\vhs\in\Xhzs$ such that $\|\vThs\|_{\Omega}+\|\vhs\|_{\boldsymbol{{\rm X}},\flat,h}\lesssim\|\rh\|_{{\rm Y},\flat,h}\leq\|\zh\|_{\mathbb{Z},\flat,h}$, where we have used the second inequality in~\eqref{lem:stab:ah.eq2} and the definition~\eqref{def:general.norm} of $\|\cdot\|_{\mathbb{Z},\flat,h}$. Summing~\eqref{eq:eqfirst} and~\eqref{eq:eqsecond}, and using the first inequality in~\eqref{lem:stab:ah.eq2}, we infer~\eqref{thm:wellposedness.eq}.
        
        To prove well-posedness, since the linear system associated to Problem~\eqref{eq:general.discrete} is square, it is sufficient to prove injectivity. Assume that $\Ah\big((\uh,\ph),(\vh,\qh)\big)=0$ for all $(\vh,\qh)\in\Zhzs$. Taking $(\vh,\qh)=(\underline{\vec{0}}_h,\qh)$ and using~\eqref{def:general.ZZhz}, we first infer that $(\uh,\ph)\in\ZZhzs$. Taking $(\vh,\qh)=(\uh,\ph)$ and $(\vh,\qh)=(\vhs,\underline{0}_h)$ and using~\eqref{thm:wellposedness.eq}, we then get
        $$\|(\uh,\ph)\|_{\mathbb{Z},\flat,h}^2\lesssim\Ah\big((\uh,\ph),(\uh,\ph)\big)+\Ah\big((\uh,\ph),(\vhs,\underline{0}_h)\big)=0,$$
        which, by Lemma~\ref{le:norm.Z.general}, eventually yields $(\uh,\ph)=(\underline{\vec{0}}_h,\underline{0}_h)$.
        
        To prove the a priori bound~\eqref{eq:apr}, we take $\zh=(\uh,\ph)$ in~\eqref{thm:wellposedness.eq} and we use~\eqref{eq:general.discrete.equiv}. We get, by the Cauchy--Schwarz inequality,
        $$\|(\uh,\ph)\|_{\mathbb{Z},\flat,h}^2\lesssim(\boldsymbol{f},\uTh)_\Omega+(\boldsymbol{f},\vThs)_\Omega\leq\|\vec{f}\|_{\Omega}\big(\|\uTh\|_{\Omega}+\|\vThs\|_{\Omega}\big).$$
        The conclusion follows from the fact that $\|\vThs\|_{\Omega}\lesssim\|(\uh,\ph)\|_{\mathbb{Z},\flat,h}$, and from the combination of Remark~\ref{rem:latitude} with the generalized discrete Weber inequality~\eqref{rmk:max.ineq.eq} of Corollary~\ref{cor:max.ineq} applied to $(\uh,\ph)$ satisfying~\eqref{rmk:max.ineq.cond} (one can easily check that $\dh$ satisfies~\eqref{eq:approx}) to bound $\|\uTh\|_{\Omega}$.
\end{proof}

\subsubsection{Error analysis}

We recall that $(\vec{u},p)\in\Hzcurl\times H^1_0(\Omega)$ denotes the unique solution to Problem~\eqref{eq:weak}. We assume from now on that $\vec{u}$ possesses the additional regularity $\vec{u}\in \boldsymbol H^1(\Omega;\mathbb{R}^3)$, and we let $\hatuh \defi \IntXhs\boldsymbol{u}\in\Xhzs$ and $\hatph\defi \IntYh p\in\Yhz$.
We define the errors
\begin{equation}\label{def:errors}
	\Xhzs\ni\erruh\defi \uh-\hatuh,
	\qquad
	\Yhz\ni\errph\defi \ph - \hatph,
\end{equation}
where $(\uh,\ph)\in\Xhzs\times\Yhz$ is the unique solution to Problem~\eqref{eq:general.discrete}.
Recalling \eqref{eq:general.discrete.equiv} and \eqref{def:general.Ah}, the errors $(\erruh,\errph)\in\Zhzs$ solve
\begin{equation}\label{eq:discrete:Ah:error}
	\Ah\big((\erruh,\errph),(\vh,\qh)\big) = \lh(\vh) + \mh\big(\qh\big) \qquad \forall\, (\vh,\qh)\in\Zhzs,
\end{equation}
where we have defined the consistency error linear forms
\begin{subequations}\label{def:consistency.errors}
	\begin{alignat}{1}\label{def:consistency.errors.lh}
		\lh(\vh) &\defi (\boldsymbol{f},\vTh)_\Omega - \ah(\hatuh,\vh) - \bh(\vh,\hatph),
		\\\label{def:consistency.errors.mh}
		\mh\big(\qh\big) & \defi \bh(\hatuh,\qh)- \dh(\hatph,\qh).
	\end{alignat}
\end{subequations}

\begin{theorem}[Energy-error estimate]\label{thm:error}
  Assume that
  $$\boldsymbol{u}\in\Hzcurl\cap \boldsymbol H^1(\Omega;\mathbb{R}^3)\cap \boldsymbol H^{k+2}(\t_h;\mathbb{R}^3),\qquad p\in H^1_0(\Omega)\cap H^{k+1}(\t_h).$$
  Then, the following holds true, with $(\erruh,\errph)\in\Zhzs$ defined by~\eqref{def:errors}:
  \begin{equation}\label{thm:error.eq}
    \|(\erruh,\errph)\|_{\mathbb{Z},\flat,h} \lesssim \left[
      \sum_{T\in\t_h} h_T^{2(k+1)}\left(|\boldsymbol{u}|_{\boldsymbol H^{k+2}(T;\mathbb{R}^3)}^2 + |p|_{H^{k+1}(T)}^2\right)
      \right]^{\nicefrac{1}{2}}.
  \end{equation}
\end{theorem}

\begin{proof}
  Since $(\erruh,\errph)\in\Zhzs$, by~\eqref{thm:wellposedness.eq} with $\zh=(\erruh,\errph)$ (notice that $(\vhs,\underline{0}_h)\in\Zhzs$ and $\|(\vhs,\underline{0}_h)\|_{\mathbb{Z},\flat,h}=\|\vhs\|_{\boldsymbol{{\rm X}},\flat,h}\lesssim\|\zh\|_{\mathbb{Z},\flat,h}$) and~\eqref{eq:discrete:Ah:error}, we infer
  \begin{equation}\label{proof:error.1}
    \|(\erruh,\errph)\|_{\mathbb{Z},\flat,h}\lesssim\max_{(\vh,\qh)\in\Zhzs,\|(\vh,\qh)\|_{\mathbb{Z},\flat,h}=1}\big(\lh(\vh) + \mh\big(\qh\big)\big).
  \end{equation}
	
	Let us first focus on $\lh(\vh)$ for $\vh\in\Xhzs$. By~\eqref{def:consistency.errors.lh}, the fact that $\vec{f}=\Curl(\Curl\vec{u})+\Grad p$ in $\Omega$, and element-by-element integration by parts, we infer
	\begin{equation}\label{proof:error.2}
          \begin{split}
	    \lh(\vh) &=  \!\!\sum_{T\in\t_h} \!\Bigg[\big(\Curl \boldsymbol{u},\Curl\vT\big)_T  -  \!\!\!\sum_{F\in\f_T} \big(\Curl \boldsymbol{u}_{\mid F}{\times}\normal_{TF},\vT_{\mid F}\big)_F \Bigg]
	    \\
	    &\qquad + \big(\Grad p,\vTh\big)_\Omega
	    -\ah(\hatuh,\vh)
	    -\bh(\vh,\hatph)
	    \\
	    &= \!\!\sum_{T\in\t_h} \!\Bigg(\big(\Curl \boldsymbol{u},\Curl\vT\big)_T
	    \\
	    &\qquad -  \!\!\!\sum_{F\in\f_T} \big(\vec{\gamma}_{\tau,F}\big(\Curl\boldsymbol{u}{\times}\normal_{TF}\big),\vec{\gamma}_{\tau,F}(\vT)-\vF\big)_F \Bigg)-\ah(\hatuh,\vh),
	  \end{split}
        \end{equation}
	where we have used the fact that the tangential component of $\Curl \boldsymbol{u}$ is continuous across interfaces (as a consequence of the fact that $\Curl\vec{u}\in\Hcurl\cap \boldsymbol H^1(\t_h;\mathbb{R}^3)$) along with $\vF=\boldsymbol{0}$ for all $F\in\f_h^{\rm b}$ to insert $\vF$ into the boundary term, together with the fact that $\big(\Grad p,\vTh\big)_\Omega = \bh(\vh,\hatph)$ as a consequence of the commutation property~\eqref{lem:commut.GTkITk.eq}.
	Using the definitions~\eqref{def:general.bf.ah} of $\ah$ and~\eqref{def:CTk} of $\CkT$ for $T\in\t_h$ (testing with $\vec{w}=\CkT\hatuTF\in\boldsymbol{{\cal R}}^k(T)$), and integrating by parts, we have
	\begin{equation}\label{proof:error.3}
          \begin{split}
		\ah(\hatuh,\vh) =  
		\hspace*{-3px}&\sum_{T\in\t_h}\hspace*{-3px}\Bigg[
			\big(\CkT \hatuTF,\Curl \vT\big)_T
		\\
                &-\sum_{F\in\f_T}\hspace*{-1px}\big(\vec{\gamma}_{\tau,F}(\CkT\underline{\widehat{\vec{\rm u}}}_T{\times}\normal_{TF}),\vec{\gamma}_{\tau,F}(\vT)-\vF\big)_F
		\Bigg] + \sh(\hatuh,\vh).
	  \end{split}
        \end{equation}
	Since, by Lemma~\ref{lemma:commut.CTkITk}, $\CkT \hatuTF=\vec{\pi}_{\boldsymbol{{\cal R}},T}^k\big(\Curl\vec{u}_{\mid T}\big)$ for all $T\in\t_h$, a combination of~\eqref{proof:error.2} and~\eqref{proof:error.3} yields (recall that $\vT\in\boldsymbol\Poly^{k+1}(T)$)
	\begin{equation*}
	  \begin{aligned}
		\lh(\vh)  &= \sum_{T\in\t_h} \sum_{F\in\f_T} \big(\vec{\gamma}_{\tau,F}\big((\vec{\pi}^k_{\boldsymbol{{\cal R}},T}(\Curl\vec{u}_{\mid T})-\Curl\vec{u}){\times}\normal_{TF}\big),\vec{\gamma}_{\tau,F}(\vT)-\vF\big)_F
		\\
		& \qquad\qquad\qquad\qquad\qquad - \sh(\hatuh,\vh).
	  \end{aligned}
	\end{equation*}
        Applying the triangle and Cauchy--Schwarz inequalities, we get
        \begin{equation}\label{proof:error.4}
          \begin{aligned}
            |\lh(\vh)|&\leq\sum_{T\in\t_h} \sum_{F\in\f_T}\|\vec{\pi}^k_{\boldsymbol{{\cal R}},T}(\Curl\vec{u}_{\mid T})-\Curl\vec{u}\|_F~\|\vec{\gamma}_{\tau,F}(\vT)-\vF\|_F\\
            & \qquad\qquad+\sh(\hatuh,\hatuh)^{\nicefrac12}\sh(\vh,\vh)^{\nicefrac12}.
          \end{aligned}
	\end{equation}
        Let us focus on $\|\vec{\pi}^k_{\boldsymbol{{\cal R}},T}(\Curl\vec{u}_{\mid T})-\Curl\vec{u}\|_F$ for $F\in\f_T$. Adding/subtracting $\Curl\big(\vec{\pi}^{k+1}_{\boldsymbol{\cal P},T}(\vec{u}_{\mid T})\big)$, using a triangle inequality, a discrete trace inequality (see, e.g.,~\cite[Lemma 1.32]{Di-Pietro.Droniou:20}) on $\vec{\pi}^k_{\boldsymbol{{\cal R}},T}(\Curl\vec{u}_{\mid T})-\Curl\big(\vec{\pi}^{k+1}_{\boldsymbol{\cal P},T}(\vec{u}_{\mid T})\big)$, and the approximation properties of $\vec{\pi}^{k+1}_{\boldsymbol{\cal P},T}$ on mesh faces (see, e.g.,~\cite[Theorem 1.45]{Di-Pietro.Droniou:20}) for $\Curl\big(\vec{\pi}^{k+1}_{\boldsymbol{\cal P},T}(\vec{u}_{\mid T})\big)-\Curl\vec{u}$, we infer
        \begin{multline*}
          h_F^{\nicefrac12}\|\vec{\pi}^k_{\boldsymbol{{\cal R}},T}(\Curl\vec{u}_{\mid T})-\Curl\vec{u}\|_F\lesssim\|\vec{\pi}^k_{\boldsymbol{{\cal R}},T}(\Curl\vec{u}_{\mid T})-\Curl\vec{u}\|_T\\+\|\Curl\vec{u}-\Curl\big(\vec{\pi}^{k+1}_{\boldsymbol{\cal P},T}(\vec{u}_{\mid T})\big)\|_T+h_T^{k+1}|\vec{u}|_{\boldsymbol H^{k+2}(T;\mathbb{R}^3)},
        \end{multline*}
        where we have used yet another triangle inequality to insert $\Curl\boldsymbol{u}$. The second term on the right-hand side is readily estimated using again the approximation properties of $\vec{\pi}^{k+1}_{\boldsymbol{\cal P},T}$. 
        As far as the first term is concerned, we have
        $$\|\vec{\pi}^k_{\boldsymbol{{\cal R}},T}(\Curl\vec{u}_{\mid T})-\Curl\vec{u}\|_T=\min_{\vec{p}\in\boldsymbol\Poly^{k+1}(T)}\|\Curl\vec{p}-\Curl\vec{u}\|_T,$$
        which finally yields
        $$h_F^{\nicefrac12}\|\vec{\pi}^k_{\boldsymbol{{\cal R}},T}(\Curl\vec{u}_{\mid T})-\Curl\vec{u}\|_F\lesssim h_T^{k+1}|\vec{u}|_{\boldsymbol H^{k+2}(T;\mathbb{R}^3)}.$$
        Plugging this last estimate into~\eqref{proof:error.4}, applying a discrete Cauchy--Schwarz inequality, and using~\eqref{eq:max.proof:I2b} as well as~\eqref{eq:hatuh} (with $\vec{\pi}^{k+1}_{\boldsymbol{{\cal P}}_\flat,F}$ instead of $\vec{\pi}^{k+1}_{\boldsymbol{{\cal G}},F}$), we infer
        \begin{equation} \label{eq:lh.gen}
          |\lh(\vh)|\lesssim\left(\sum_{T\in\t_h}h_T^{2(k+1)}|\vec{u}|^2_{\boldsymbol H^{k+2}(T;\mathbb{R}^3)}\right)^{\nicefrac12}\,\|\vh\|_{\boldsymbol{\rm X},\flat,h}.
        \end{equation}
	
	Let us now focus on $\mh\big(\qh\big)$ for $\qh\in\Yhz$.
        Recalling~\eqref{eq:bh.hatuh} and \eqref{def:standard.consistency.errors.mh}, one can readily infer
        \begin{equation} \label{eq.bh.hatuh.gen}
          |\bh(\hatuh,\qh)|
          \lesssim\left(\sum_{T\in\t_h}h_T^{2(k+1)}|\vec{u}|^2_{\boldsymbol H^{k+2}(T;\mathbb{R}^3)}\right)^{\nicefrac12}\,\dh(\qh,\qh)^{\nicefrac12}.
        \end{equation}
        Now, applying the Cauchy--Schwarz inequality, recalling the definition~\eqref{def:general.bf.ch} of $\dh$, noticing that $\pi_{{\cal P},T}^{k}(p_{\mid T})_{\mid F}=\pi_{{\cal P},F}^{k+1}\big(\pi_{{\cal P},T}^{k}(p_{\mid T})_{\mid F}\big)$ for all $T\in\t_h$ and $F\in\f_T$, and using the $L^2(F)$-boundedness of $\pi^{k+1}_{{\cal P},F}$, we infer
        \begin{align*}
          |\dh(\hatph,\qh)|
          &\leq\dh(\hatph,\hatph)^{\nicefrac12}\,\dh(\qh,\qh)^{\nicefrac12}
          \\
          &\leq\left(\sum_{T\in\t_h}\sum_{F\in\f_T}h_F\|p-\pi_{{\cal P},T}^k(p_{\mid T})\|_F^2\right)^{\nicefrac12}\,\dh(\qh,\qh)^{\nicefrac12}.
        \end{align*}
        By the approximation properties of $\pi^k_{{\cal P},T}$ on mesh faces (see, e.g.,~\cite[Theorem 1.45]{Di-Pietro.Droniou:20}), we get
        \begin{equation}\label{ch.hatph.gen}
          |\dh(\hatph,\qh)|
          \lesssim\left(\sum_{T\in\t_h}h_T^{2(k+1)}|p|^2_{H^{k+1}(T)}\right)^{\nicefrac12}\,\dh(\qh,\qh)^{\nicefrac12}.
        \end{equation}
        Gathering~\eqref{eq.bh.hatuh.gen}--\eqref{ch.hatph.gen}, and recalling the definition~\eqref{def:consistency.errors.mh} of $\mh\big(\qh\big)$, finally yields
        \begin{equation}\label{eq:mh.gen}
          |\mh\big(\qh\big)|
          \lesssim\left(\sum_{T\in\t_h} h_T^{2(k+1)}\left(|\boldsymbol{u}|_{\boldsymbol H^{k+2}(T;\mathbb{R}^3)}^2 + |p|_{H^{k+1}(T)}^2\right)\right)^{\nicefrac{1}{2}}\|\qh\|_{{\rm Y},\flat,h}.
        \end{equation}
        Plugging~\eqref{eq:lh.gen} and~\eqref{eq:mh.gen} into~\eqref{proof:error.1} for $(\vh,\qh)$ such that $\|(\vh,\qh)\|_{\mathbb{Z},\flat,h}=1$ finally yields~\eqref{thm:error.eq}.
\end{proof}

\begin{remark}[The tetrahedral case]\label{rk:tet.case}
  On matching tetrahedral meshes, according to Lemma~\ref{le:GTk.norm}, $\|\Gkh\cdot\|_{\Omega}$ defines a norm on $\Yhz$. Hence, in this case, and as already pointed out in~\ifMMMAS\else\cite{Chen.ea:17}\fi, one can consider a modified version of Problem~\eqref{eq:general.discrete} in which $\dh$ is removed and for which stability (hence well-posedness) is preserved. This is a direct consequence of Lemma~\ref{lem:infsupbh}, which essentially states an inf-sup condition for $\bh$ provided $\|\Gkh\cdot\|_{\Omega}$ is a norm on $\Yhz$. Removing $\dh$, $\uh\in\XXhzs$ holds true and, as a by-product of the commutation property~\eqref{lem:commut.GTkITk.eq}, $\uTh\in\Hdivz$. Furthermore, a close inspection of the proof of Theorem~\ref{thm:error} shows that, in this case, one ends up with an energy-error estimate that is free of Lagrange multiplier contribution. This allows one to reproduce at the discrete level the following structure of Problem~\eqref{eq:strong}. When the current density $\vec{f}$ is given by the gradient of some function $\psi\in H^1_0(\Omega)$, then $\vec{u}=\vec{0}$ and $p=\psi$. At the discrete level, when $\vec{f}=\Grad\psi$, one then gets $\uh=\underline{\vec{0}}_h$ and $\ph=\IntYh\psi$ (note that this can be observed in practice up to machine precision only if the computation of the right-hand side is also performed up to machine precision).
\end{remark}

\subsubsection{Numerical results}\label{sec:num}

Let the domain $\Omega$ be the unit cube $(0,1)^3$. We consider Problem~\eqref{eq:strong} with exact solution
$$
\boldsymbol{u}(x_1,x_2,x_3) \defi \left(
\begin{tabular}{l}
$\sin(\pi x_2)\sin(\pi x_3)$
\\
$\sin(\pi x_1)\sin(\pi x_3)$
\\
$\sin(\pi x_1)\sin(\pi x_2)$
\end{tabular}
\right), 
\;
p(x_1,x_2,x_3) \defi \sin(\pi x_1)\sin(\pi x_2)\sin(\pi x_3).
$$
Clearly, the magnetic vector potential $\boldsymbol{u}$ and the Lagrange multiplier $p$ satisfy~\eqref{eq:strong:2},~\eqref{eq:strong:bc.u}, and~\eqref{eq:strong:bc.p}.
The current density $\boldsymbol{f}$ is set according to~\eqref{eq:strong:1}.

As in Section~\ref{sec:standard:num}, we solve the discrete Problem~\eqref{eq:general.discrete} on two refined mesh sequences, of respectively cubic and regular tetrahedral meshes. 
For each problem, the element unknowns for both the magnetic vector potential and the Lagrange multiplier are locally eliminated using a Schur complement technique. This step is fully parallelizable. The resulting (condensed) global linear system is solved using the SparseLU direct solver of the Eigen library on the same architecture as in Section \ref{sec:standard:num}.
For $k\in\{0,1,2\}$, we depict on Figures~\ref{fig:general.error.cubic} and~\ref{fig:general.error.tetra_nodh}, respectively for the cubic and (regular) tetrahedral mesh families, the relative energy-error $\|\uh-\IntXhs\boldsymbol{u}\|_{\boldsymbol{\rm X},\flat,h}/\|\IntXhs\boldsymbol{u}\|_{\boldsymbol{\rm X},\flat,h}$ (top row), $L^2$-error $\|\uTh-\boldsymbol{\pi}_{\boldsymbol{\cal P},h}^{k+1}\boldsymbol{u}\|_{\Omega}/\|\boldsymbol{\pi}_{\boldsymbol{\cal P},h}^{k+1}\boldsymbol{u}\|_{\Omega}$ (middle row), and $L^2$-like-error $\|\ph -\IntYh p\|_{{\rm Y},\flat,h}/\|\IntYh p\|_{{\rm Y},\flat,h}$ (bottom row) as functions of
\begin{inparaenum}[(i)]
	\item the meshsize (left column),
	\item the solution time in seconds, i.e.~the time needed to solve the (condensed) global linear system (center column), and
	\item the number of (interface) DoF (right column).
\end{inparaenum}
For the two mesh families, we obtain, as predicted by Theorem~\ref{thm:error}, a convergence rate of the energy-error on the magnetic vector potential and of the $L^2$-like-error on the Lagrange multiplier of order $k+1$. We also observe a convergence rate of order $k+2$ for the $L^2$-error on the magnetic vector potential.
Following Remark~\ref{rk:tet.case}, we also solve on the (matching) tetrahedral mesh family a modified version of Problem~\eqref{eq:general.discrete} in which $\dh$ is removed, and we display on Figure~\ref{fig:general.error.tetra_nodh} the results in dashed lines. We plot the relative energy-error $\|\uh-\IntXhs\boldsymbol{u}\|_{\boldsymbol{\rm X},\flat,h}/\|\IntXhs\boldsymbol{u}\|_{\boldsymbol{\rm X},\flat,h}$ (top row), the $L^2$-error $\|\uTh-\boldsymbol{\pi}_{\boldsymbol{\cal P},h}^{k+1}\boldsymbol{u}\|_{\Omega}/\|\boldsymbol{\pi}_{\boldsymbol{\cal P},h}^{k+1}\boldsymbol{u}\|_{\Omega}$ (middle row), and since we remove the contribution $\dh$, we replace the $L^2$-like-error measure $\|\ph -\IntYh p\|_{{\rm Y},\flat,h}$ (bottom row) by the measure
  $$\left(\sum_{T\in\t_h}h_T^2\|\GkT\big(\pTF - \underline{\rm I}_{{\rm Y},T}^{k+1}(p_{\mid T})\big)\|_T^2\right)^{\nicefrac12},$$
which remains meaningful for $k=0$.  
Also in this case, we obtain the predicted energy-error on the magnetic vector potential and $L^2$-like-error on the Lagrange multiplier of order $k+1$, and we observe a convergence rate of order $k+2$ for the $L^2$-error on the magnetic vector potential.
The numerical tests show that removing $\dh$ leads to a slightly more accurate approximation of the magnetic vector potential (and a slightly less accurate approximation of the Lagrange multiplier, but this latter result is not meaningful since two different error measures are used).

\begin{figure}\centering
  \ref{legend:hexa:2}
  \vspace{0.50cm}\\
  \begin{minipage}[b]{0.30\textwidth}
    \ifMMMAS 
    \begin{tikzpicture}[scale=0.55]
    \else
    \begin{tikzpicture}[scale=0.60]
    \fi
      \begin{loglogaxis}[legend columns=-1, legend to name=legend:hexa:2]
        \addplot table[x=meshsize,y=errXnorm_u] {cv_potential/k0_hex.dat};
        \addplot table[x=meshsize,y=errXnorm_u] {cv_potential/k1_hex.dat};
        \addplot table[x=meshsize,y=errXnorm_u] {cv_potential/k2_hex.dat};
        \logLogSlopeTriangle{0.90}{0.4}{0.1}{1}{black};
        \logLogSlopeTriangle{0.90}{0.4}{0.1}{2}{black};
        \logLogSlopeTriangle{0.90}{0.4}{0.1}{3}{black};
        \legend{$k=0$,$k=1$,$k=2$};
      \end{loglogaxis}
    \end{tikzpicture}
  \end{minipage}
  \hspace{0.025\textwidth}
  \begin{minipage}[b]{0.30\textwidth}
    \ifMMMAS 
    \begin{tikzpicture}[scale=0.55]
    \else
    \begin{tikzpicture}[scale=0.60]
    \fi
      \begin{loglogaxis}
        \addplot table[x=tot_solution_time,y=errXnorm_u] {cv_potential/k0_hex.dat};
        \addplot table[x=tot_solution_time,y=errXnorm_u] {cv_potential/k1_hex.dat};
        \addplot table[x=tot_solution_time,y=errXnorm_u] {cv_potential/k2_hex.dat};
      \end{loglogaxis}
    \end{tikzpicture}
  \end{minipage}
  \hspace{0.025\textwidth}
  \begin{minipage}[b]{0.30\textwidth}
    \ifMMMAS 
    \begin{tikzpicture}[scale=0.55]
    \else
    \begin{tikzpicture}[scale=0.60]
    \fi
      \begin{loglogaxis}
        \addplot table[x=n_DOFs,y=errXnorm_u]{cv_potential/k0_hex.dat};
        \addplot table[x=n_DOFs,y=errXnorm_u]{cv_potential/k1_hex.dat};
        \addplot table[x=n_DOFs,y=errXnorm_u]{cv_potential/k2_hex.dat};
        \logLogSlopeTriangleNDOFs{0.10}{-0.4}{0.1}{1/3}{black};
        \logLogSlopeTriangleNDOFs{0.10}{-0.4}{0.1}{2/3}{black};
        \logLogSlopeTriangleNDOFs{0.10}{-0.4}{0.1}{1}{black};
      \end{loglogaxis}
    \end{tikzpicture}
  \end{minipage}
  \vspace{0.25cm}\\
  \begin{minipage}[b]{0.30\textwidth}
    \ifMMMAS 
    \begin{tikzpicture}[scale=0.55]
    \else
    \begin{tikzpicture}[scale=0.60]
    \fi
      \begin{loglogaxis}
        \addplot table[x=meshsize,y=errL2_u] {cv_potential/k0_hex.dat};
        \addplot table[x=meshsize,y=errL2_u] {cv_potential/k1_hex.dat};
        \addplot table[x=meshsize,y=errL2_u] {cv_potential/k2_hex.dat};
        \logLogSlopeTriangle{0.90}{0.4}{0.1}{2}{black};
        \logLogSlopeTriangle{0.90}{0.4}{0.1}{3}{black};
        \logLogSlopeTriangle{0.90}{0.4}{0.1}{4}{black};
      \end{loglogaxis}
    \end{tikzpicture}
  \end{minipage}
  \hspace{0.025\textwidth}
  \begin{minipage}[b]{0.30\textwidth}
    \ifMMMAS 
    \begin{tikzpicture}[scale=0.55]
    \else
    \begin{tikzpicture}[scale=0.60]
    \fi
      \begin{loglogaxis}
        \addplot table[x=tot_solution_time,y=errL2_u] {cv_potential/k0_hex.dat};
        \addplot table[x=tot_solution_time,y=errL2_u] {cv_potential/k1_hex.dat};
        \addplot table[x=tot_solution_time,y=errL2_u] {cv_potential/k2_hex.dat};
      \end{loglogaxis}
    \end{tikzpicture}
  \end{minipage}
  \hspace{0.025\textwidth}
  \begin{minipage}[b]{0.30\textwidth}
    \ifMMMAS 
    \begin{tikzpicture}[scale=0.55]
    \else
    \begin{tikzpicture}[scale=0.60]
    \fi
      \begin{loglogaxis}
        \addplot table[x=n_DOFs,y=errL2_u]{cv_potential/k0_hex.dat};
        \addplot table[x=n_DOFs,y=errL2_u]{cv_potential/k1_hex.dat};
        \addplot table[x=n_DOFs,y=errL2_u]{cv_potential/k2_hex.dat};
        \logLogSlopeTriangleNDOFs{0.10}{-0.4}{0.1}{2/3}{black};
        \logLogSlopeTriangleNDOFs{0.10}{-0.4}{0.1}{1}{black};
        \logLogSlopeTriangleNDOFs{0.10}{-0.4}{0.1}{4/3}{black};
      \end{loglogaxis}
    \end{tikzpicture}
  \end{minipage}
  \vspace{0.25cm}\\
  \begin{minipage}[b]{0.30\textwidth}
    \ifMMMAS 
    \begin{tikzpicture}[scale=0.55]
    \else
    \begin{tikzpicture}[scale=0.60]
    \fi
      \begin{loglogaxis}
        \addplot table[x=meshsize,y=errYnorm_p] {cv_potential/k0_hex.dat};
        \addplot table[x=meshsize,y=errYnorm_p] {cv_potential/k1_hex.dat};
        \addplot table[x=meshsize,y=errYnorm_p] {cv_potential/k2_hex.dat};
        \logLogSlopeTriangle{0.90}{0.4}{0.1}{1}{black};
        \logLogSlopeTriangle{0.90}{0.4}{0.1}{2}{black};
        \logLogSlopeTriangle{0.90}{0.4}{0.1}{3}{black};
      \end{loglogaxis}
    \end{tikzpicture}
  \end{minipage}
  \hspace{0.025\textwidth}
  \begin{minipage}[b]{0.30\textwidth}
    \ifMMMAS 
    \begin{tikzpicture}[scale=0.55]
    \else
    \begin{tikzpicture}[scale=0.60]
    \fi
      \begin{loglogaxis}
        \addplot table[x=tot_solution_time,y=errYnorm_p] {cv_potential/k0_hex.dat};
        \addplot table[x=tot_solution_time,y=errYnorm_p] {cv_potential/k1_hex.dat};
        \addplot table[x=tot_solution_time,y=errYnorm_p] {cv_potential/k2_hex.dat};
      \end{loglogaxis}
    \end{tikzpicture}
  \end{minipage}
  \hspace{0.025\textwidth}
  \begin{minipage}[b]{0.30\textwidth}
    \ifMMMAS 
    \begin{tikzpicture}[scale=0.55]
    \else
    \begin{tikzpicture}[scale=0.60]
    \fi
      \begin{loglogaxis}
        \addplot table[x=n_DOFs,y=errYnorm_p]{cv_potential/k0_hex.dat};
        \addplot table[x=n_DOFs,y=errYnorm_p]{cv_potential/k1_hex.dat};
        \addplot table[x=n_DOFs,y=errYnorm_p]{cv_potential/k2_hex.dat};
        \logLogSlopeTriangleNDOFs{0.10}{-0.4}{0.1}{1/3}{black};
        \logLogSlopeTriangleNDOFs{0.10}{-0.4}{0.1}{2/3}{black};
        \logLogSlopeTriangleNDOFs{0.10}{-0.4}{0.1}{1}{black};
      \end{loglogaxis}
    \end{tikzpicture}
  \end{minipage}
  \caption{\label{fig:general.error.cubic}
    Relative energy-error ${\|\uh-\IntXhs\boldsymbol{u}\|_{\boldsymbol{\rm X},\flat,h}}/{\|\IntXhs\boldsymbol{u}\|_{\boldsymbol{\rm X},\flat,h}}$ (\emph{top row}),
    $L^2$-error ${\|\uTh-\boldsymbol{\pi}_{\boldsymbol{\cal P},h}^{k+1}\boldsymbol{u}\|_{\Omega}}/{\|\boldsymbol{\pi}_{\boldsymbol{\cal P},h}^{k+1}\boldsymbol{u}\|_{\Omega}}$ (\emph{middle row}),
    and $L^2$-like-error ${\|\ph -\IntYh p\|_{{\rm Y},\flat,h}}/$ ${\|\IntYh p\|_{{\rm Y},\flat,h}}$ (\emph{bottom row})
    versus meshsize $h$ (\emph{left column}), solution time (\emph{center column}), and number of DoF (\emph{right column}) on cubic meshes for the test-case of Section~\ref{sec:num}.}
\end{figure}

\begin{figure}\centering
  \ref{legend:tria:2nostab}
  \vspace{0.50cm}\\
  \begin{minipage}[b]{0.30\textwidth}
    \ifMMMAS 
    \begin{tikzpicture}[scale=0.55]
    \else
    \begin{tikzpicture}[scale=0.60]
    \fi
      \begin{loglogaxis}[legend columns=-1, legend to name=legend:tria:2nostab]
        \addplot table[x=meshsize,y=errXnorm_u]{cv_potential/k0_tet_dh.dat};
        \addplot table[x=meshsize,y=errXnorm_u]{cv_potential/k1_tet_dh.dat};
        \addplot table[x=meshsize,y=errXnorm_u]{cv_potential/k2_tet_dh.dat};
        \addplot[dashed,color=blue,mark=+] table[x=meshsize,y=errXnorm_u] {cv_potential/k0_tet_nostab.dat};
        \addplot[dashed,color=red,mark=+] table[x=meshsize,y=errXnorm_u] {cv_potential/k1_tet_nostab.dat};
        \addplot[dashed,color=brown,mark=+] table[x=meshsize,y=errXnorm_u] {cv_potential/k2_tet_nostab.dat};
        \logLogSlopeTriangle{0.90}{0.4}{0.1}{1}{black};
        \logLogSlopeTriangle{0.90}{0.4}{0.1}{2}{black};
        \logLogSlopeTriangle{0.90}{0.4}{0.1}{3}{black};
        \legend{$k=0$,$k=1$,$k=2$};
      \end{loglogaxis}
    \end{tikzpicture}
  \end{minipage}
  \hspace{0.025\textwidth}
  \begin{minipage}[b]{0.30\textwidth}
    \ifMMMAS 
    \begin{tikzpicture}[scale=0.55]
    \else
    \begin{tikzpicture}[scale=0.60]
    \fi
      \begin{loglogaxis}   
        \addplot table[x=tot_solution_time,y=errXnorm_u]{cv_potential/k0_tet_dh.dat};
        \addplot table[x=tot_solution_time,y=errXnorm_u]{cv_potential/k1_tet_dh.dat};
        \addplot table[x=tot_solution_time,y=errXnorm_u]{cv_potential/k2_tet_dh.dat};
        \addplot[dashed,color=blue,mark=+] table[x=tot_solution_time,y=errXnorm_u] {cv_potential/k0_tet_nostab.dat};
        \addplot[dashed,color=red,mark=+] table[x=tot_solution_time,y=errXnorm_u] {cv_potential/k1_tet_nostab.dat};
        \addplot[dashed,color=brown,mark=+] table[x=tot_solution_time,y=errXnorm_u] {cv_potential/k2_tet_nostab.dat}; 
      \end{loglogaxis}
    \end{tikzpicture}
  \end{minipage}
  \hspace{0.025\textwidth}
  \begin{minipage}[b]{0.30\textwidth}
    \ifMMMAS 
    \begin{tikzpicture}[scale=0.55]
    \else
    \begin{tikzpicture}[scale=0.60]
    \fi
      \begin{loglogaxis}
        \addplot table[x=n_DOFs,y=errXnorm_u]{cv_potential/k0_tet_dh.dat};
        \addplot table[x=n_DOFs,y=errXnorm_u]{cv_potential/k1_tet_dh.dat};
        \addplot table[x=n_DOFs,y=errXnorm_u]{cv_potential/k2_tet_dh.dat};
        \addplot[dashed,color=blue,mark=+] table[x=n_DOFs,y=errXnorm_u]{cv_potential/k0_tet_nostab.dat};
        \addplot[dashed,color=red,mark=+] table[x=n_DOFs,y=errXnorm_u]{cv_potential/k1_tet_nostab.dat};
        \addplot[dashed,color=brown,mark=+] table[x=n_DOFs,y=errXnorm_u]{cv_potential/k2_tet_nostab.dat};
        \logLogSlopeTriangleNDOFs{0.10}{-0.4}{0.1}{1/3}{black};
        \logLogSlopeTriangleNDOFs{0.10}{-0.4}{0.1}{2/3}{black};
        \logLogSlopeTriangleNDOFs{0.10}{-0.4}{0.1}{1}{black};
      \end{loglogaxis}
    \end{tikzpicture}
  \end{minipage}
  \vspace{0.25cm}\\
  \begin{minipage}[b]{0.30\textwidth}
    \ifMMMAS 
    \begin{tikzpicture}[scale=0.55]
    \else
    \begin{tikzpicture}[scale=0.60]
    \fi
      \begin{loglogaxis}
        \addplot table[x=meshsize,y=errL2_u]{cv_potential/k0_tet_dh.dat};
        \addplot table[x=meshsize,y=errL2_u]{cv_potential/k1_tet_dh.dat};
        \addplot table[x=meshsize,y=errL2_u]{cv_potential/k2_tet_dh.dat};
        \addplot[dashed,color=blue,mark=+] table[x=meshsize,y=errL2_u] {cv_potential/k0_tet_nostab.dat};
        \addplot[dashed,color=red,mark=+] table[x=meshsize,y=errL2_u] {cv_potential/k1_tet_nostab.dat};
        \addplot[dashed,color=brown,mark=+] table[x=meshsize,y=errL2_u] {cv_potential/k2_tet_nostab.dat};
        \logLogSlopeTriangle{0.90}{0.4}{0.1}{2}{black};
        \logLogSlopeTriangle{0.90}{0.4}{0.1}{3}{black};
        \logLogSlopeTriangle{0.90}{0.4}{0.1}{4}{black};
      \end{loglogaxis}
    \end{tikzpicture}
  \end{minipage}
  \hspace{0.025\textwidth}
  \begin{minipage}[b]{0.30\textwidth}
    \ifMMMAS 
    \begin{tikzpicture}[scale=0.55]
    \else
    \begin{tikzpicture}[scale=0.60]
    \fi
      \begin{loglogaxis}
        \addplot table[x=tot_solution_time,y=errL2_u]{cv_potential/k0_tet_dh.dat};
        \addplot table[x=tot_solution_time,y=errL2_u]{cv_potential/k1_tet_dh.dat};
        \addplot table[x=tot_solution_time,y=errL2_u]{cv_potential/k2_tet_dh.dat};
        \addplot[dashed,color=blue,mark=+] table[x=tot_solution_time,y=errL2_u] {cv_potential/k0_tet_nostab.dat};
        \addplot[dashed,color=red,mark=+] table[x=tot_solution_time,y=errL2_u] {cv_potential/k1_tet_nostab.dat};
        \addplot[dashed,color=brown,mark=+] table[x=tot_solution_time,y=errL2_u] {cv_potential/k2_tet_nostab.dat};
      \end{loglogaxis}
    \end{tikzpicture}
  \end{minipage}
  \hspace{0.025\textwidth}
  \begin{minipage}[b]{0.30\textwidth}
    \ifMMMAS 
    \begin{tikzpicture}[scale=0.55]
    \else
    \begin{tikzpicture}[scale=0.60]
    \fi
      \begin{loglogaxis}
        \addplot table[x=n_DOFs,y=errL2_u]{cv_potential/k0_tet_dh.dat};
        \addplot table[x=n_DOFs,y=errL2_u]{cv_potential/k1_tet_dh.dat};
        \addplot table[x=n_DOFs,y=errL2_u]{cv_potential/k2_tet_dh.dat};
        \addplot[dashed,color=blue,mark=+] table[x=n_DOFs,y=errL2_u]{cv_potential/k0_tet_nostab.dat};
        \addplot[dashed,color=red,mark=+] table[x=n_DOFs,y=errL2_u]{cv_potential/k1_tet_nostab.dat};
        \addplot[dashed,color=brown,mark=+] table[x=n_DOFs,y=errL2_u]{cv_potential/k2_tet_nostab.dat};
        \logLogSlopeTriangleNDOFs{0.10}{-0.4}{0.1}{2/3}{black};
        \logLogSlopeTriangleNDOFs{0.10}{-0.4}{0.1}{1}{black};
        \logLogSlopeTriangleNDOFs{0.10}{-0.4}{0.1}{4/3}{black};
      \end{loglogaxis}
    \end{tikzpicture}
  \end{minipage}
  \vspace{0.25cm}\\
  \begin{minipage}[b]{0.30\textwidth}
    \ifMMMAS 
    \begin{tikzpicture}[scale=0.55]
    \else
    \begin{tikzpicture}[scale=0.60]
    \fi
      \begin{loglogaxis}
        \addplot table[x=meshsize,y=errGTStab_p]{cv_potential/k0_tet_dh.dat};
        \addplot table[x=meshsize,y=errGTStab_p]{cv_potential/k1_tet_dh.dat};
        \addplot table[x=meshsize,y=errGTStab_p]{cv_potential/k2_tet_dh.dat};
        \addplot[dashed,color=blue,mark=+] table[x=meshsize,y=errGTStab_p] {cv_potential/k0_tet_nostab.dat};
        \addplot[dashed,color=red,mark=+] table[x=meshsize,y=errGTStab_p] {cv_potential/k1_tet_nostab.dat};
        \addplot[dashed,color=brown,mark=+] table[x=meshsize,y=errGTStab_p] {cv_potential/k2_tet_nostab.dat};
        \logLogSlopeTriangle{0.90}{0.4}{0.1}{1}{black};
        \logLogSlopeTriangle{0.90}{0.4}{0.1}{2}{black};
        \logLogSlopeTriangle{0.90}{0.4}{0.1}{3}{black};
      \end{loglogaxis}
    \end{tikzpicture}
  \end{minipage}
  \hspace{0.025\textwidth}
  \begin{minipage}[b]{0.30\textwidth}
    \ifMMMAS 
    \begin{tikzpicture}[scale=0.55]
    \else
    \begin{tikzpicture}[scale=0.60]
    \fi
      \begin{loglogaxis}
        \addplot table[x=tot_solution_time,y=errGTStab_p]{cv_potential/k0_tet_dh.dat};
        \addplot table[x=tot_solution_time,y=errGTStab_p]{cv_potential/k1_tet_dh.dat};
        \addplot table[x=tot_solution_time,y=errGTStab_p]{cv_potential/k2_tet_dh.dat};
        \addplot[dashed,color=blue,mark=+] table[x=tot_solution_time,y=errGTStab_p] {cv_potential/k0_tet_nostab.dat};
        \addplot[dashed,color=red,mark=+] table[x=tot_solution_time,y=errGTStab_p] {cv_potential/k1_tet_nostab.dat};
        \addplot[dashed,color=brown,mark=+] table[x=tot_solution_time,y=errGTStab_p] {cv_potential/k2_tet_nostab.dat};
      \end{loglogaxis}
    \end{tikzpicture}
  \end{minipage}
  \hspace{0.025\textwidth}
  \begin{minipage}[b]{0.30\textwidth}
    \ifMMMAS 
    \begin{tikzpicture}[scale=0.55]
    \else
    \begin{tikzpicture}[scale=0.60]
    \fi
      \begin{loglogaxis}
        \addplot table[x=n_DOFs,y=errGTStab_p]{cv_potential/k0_tet_dh.dat};
        \addplot table[x=n_DOFs,y=errGTStab_p]{cv_potential/k1_tet_dh.dat};
        \addplot table[x=n_DOFs,y=errGTStab_p]{cv_potential/k2_tet_dh.dat};
        \addplot[dashed,color=blue,mark=+] table[x=n_DOFs,y=errGTStab_p]{cv_potential/k0_tet_nostab.dat};
        \addplot[dashed,color=red,mark=+] table[x=n_DOFs,y=errGTStab_p]{cv_potential/k1_tet_nostab.dat};
        \addplot[dashed,color=brown,mark=+] table[x=n_DOFs,y=errGTStab_p]{cv_potential/k2_tet_nostab.dat};
        \logLogSlopeTriangleNDOFs{0.10}{-0.4}{0.1}{1/3}{black};
        \logLogSlopeTriangleNDOFs{0.10}{-0.4}{0.1}{2/3}{black};
        \logLogSlopeTriangleNDOFs{0.10}{-0.4}{0.1}{1}{black};
      \end{loglogaxis}
    \end{tikzpicture}
  \end{minipage}
  \caption{\label{fig:general.error.tetra_nodh}
    Relative energy-error ${\|\uh-\IntXhs\boldsymbol{u}\|_{\boldsymbol{\rm X},\flat,h}}/{\|\IntXhs\boldsymbol{u}\|_{\boldsymbol{\rm X},\flat,h}}$ (\emph{top row}),
    $L^2$-error ${\|\uTh-\boldsymbol{\pi}_{\boldsymbol{\cal P},h}^{k+1}\boldsymbol{u}\|_{\Omega}}/{\|\boldsymbol{\pi}_{\boldsymbol{\cal P},h}^{k+1}\boldsymbol{u}\|_{\Omega}}$ (\emph{middle row}),
    and $L^2$-like-error ${\|\ph -\IntYh p\|_{{\rm Y},\flat,h}}/$ ${\|\IntYh p\|_{{\rm Y},\flat,h}}$ (\emph{bottom row})
    versus meshsize $h$ (\emph{left column}), solution time (\emph{center column}), and number of DoF (\emph{right column}) on tetrahedral meshes for the test-case of Section~\ref{sec:num}, and comparison (dashed lines) with the case where the stabilization bilinear form $\dh$ is removed.}
\end{figure}

\section*{Acknowledgments}
The authors thank Lorenzo Botti (University of Bergamo) for giving them access to his 3D \verb!C++! code \verb!SpaFEDTe! (\url{https://github.com/SpaFEDTe/spafedte.github.com}). Florent Chave and Simon Lemaire also acknowledge support from the LabEx CEMPI (ANR-11-LABX-0007-01).

\ifMMMAS
\else

\bibliographystyle{plain}

\fi

\bibliography{magnethho}

\end{document}